\documentclass[10pt]{article}
\usepackage{todonotes}
\usepackage{palatino,url}
\usepackage{verbatim,amsmath, amsthm, amssymb}

\usepackage{mathrsfs}
\usepackage{graphics,amsopn,amstext,amsfonts,color}
\usepackage{bm}
\usepackage{fullpage}
\usepackage{setspace}
\usepackage{graphicx}
\usepackage{epstopdf}
\usepackage[numbers]{natbib}
\usepackage{caption}
\usepackage{multirow}
\usepackage{booktabs}
\usepackage{paralist}
\usepackage{colortbl}
\usepackage{appendix}
\usepackage{soul}

\usepackage{xcolor}
\usepackage[bookmarksnumbered=true]{hyperref}
\usepackage{graphics,latexsym,amsfonts,verbatim,amsmath}
\usepackage{algorithm}
\usepackage{algorithmicx}
\usepackage{algpseudocode}
\usepackage{tikz}
\usetikzlibrary{calc}
\usetikzlibrary{plotmarks}
\usepackage{subcaption}

\usepackage{pgfplots}

\usepackage{pgf}
\usepackage{pst-func}

\definecolor{pinegreen}{rgb}{0.0, 0.47, 0.44}
\usepackage{pgfpages}
\usepackage{todonotes}

\usepackage{empheq}
\usepackage{cases}

\def\E{{\mathbb E}}

\def\Pr{{\mathbb{P}}}

\def\Re{\mathbb{R}}
\def\Qe{\mathbb{Q}}
\def\I{\mathbb{I}}

\def\hat{\widehat}

\def \bxi{\boldsymbol{\xi}}
\def \bzeta{\tilde{\bm{\zeta}}}

\def \P{\mathcal{P}}

\def \R{\mathcal{R}}

\def \Z{{\mathcal{Z}}}

\def\Ie{{\mathcal I}}

\def\P{{\mathcal P}}

\def\C{{\mathcal C}}

\def\Re{{\mathbb R}}

\newcommand{\exclude}[1]{}
\newcommand{\rxi}{\bm{\xi}}
 \newcommand{\y}{\bm{y}}

\newcommand{\e}{\bm{e}}

\def\CVaR{{\bf{CVaR}}}
\def\VaR{{\bf{VaR}}}

\newcommand{\zsetd}{Z}

\DeclareMathOperator{\conv}{conv}

\def\bl{\textcolor{black}}

\newcommand{\qedA}{$\Diamond$}

\newtheorem{theorem}{Theorem}
\newtheorem{claim1}{Claim}
\newtheorem{lemma}{Lemma}
\newtheorem{corollary}{Corollary}

\newtheorem{proposition}{Proposition}

\theoremstyle{definition}

\newtheorem{definition}{Definition}
\newtheorem{example}{Example}

\theoremstyle{remark}
\newtheorem{remark}{Remark}
\usepackage{authblk}

\newcommand*{\TickSize}{1pt}
\usetikzlibrary{patterns}

\newlength{\hatchspread}
\newlength{\hatchthickness}
\newlength{\hatchshift}

\tikzset{hatchspread/.code={\setlength{\hatchspread}{#1}},
         hatchthickness/.code={\setlength{\hatchthickness}{#1}},
         hatchshift/.code={\setlength{\hatchshift}{#1}},
         hatchcolor/.code={\renewcommand{\hatchcolor}{#1}}}

\author[1]{Weijun Xie\thanks{Email: wxie@vt.edu.}}
\affil[1]{Department of Industrial and Systems Engineering\\ Virginia Tech, Blacksburg, VA 24061}

\title{On Distributionally Robust Chance Constrained Programs with Wasserstein Distance}
\date{\today}
\allowdisplaybreaks

\makeatletter
\pgfmathdeclarefunction{erf}{1}{%
\begingroup
\pgfmathparse{#1 > 0 ? 1 : -1}%
\edef\sign{\pgfmathresult}%
\pgfmathparse{abs(#1)}%
\edef\x{\pgfmathresult}%
\pgfmathparse{1/(1+0.3275911*\x)}%
\edef\t{\pgfmathresult}%
\pgfmathparse{%
1 - (((((1.061405429*\t -1.453152027)*\t) + 1.421413741)*\t
-0.284496736)*\t + 0.254829592)*\t*exp(-(\x*\x))}%
\edef\y{\pgfmathresult}%
\pgfmathparse{(\sign)*\y}%
\pgfmath@smuggleone\pgfmathresult%
\endgroup
}
\makeatother

\begin{document}
\maketitle

\begin{abstract}
This paper studies a distributionally robust chance constrained program (DRCCP) with Wasserstein ambiguity set, where the uncertain constraints should {be satisfied} with a probability at least a given threshold for all the probability distributions of the uncertain parameters within a chosen Wasserstein distance from an empirical distribution. In this work, we investigate equivalent reformulations and approximations of such problems. We first show that a DRCCP can be reformulated as a conditional value-at-risk constrained optimization problem, and thus admits tight inner and outer approximations.  We also show that a DRCCP of bounded feasible region is mixed integer representable by introducing big-M coefficients and additional binary variables. For a DRCCP with pure binary decision variables, by exploring the submodular structure, we show that it admits a big-M free formulation, \bl{which} can be solved by {a} branch and cut algorithm. 
Finally, we present a numerical study to illustrate
the effectiveness of the proposed formulations.
\end{abstract}

\newpage
\section{Introduction}
\subsection{Setting}
\bl{We study distributional robust chance constrained programs (DRCCPs) of the form}:
\begin{subequations}\label{sp}
\begin{align}
\min & \ \ \bm{c}^{\top}\bm{x}, \label{sp-obj} \\
\rm{s.t. } & \ \ \bm{x} \in S, \label{sp-det} \\
& \ \ \inf_{\Pr\in \P}\Pr\left\{\tilde{\rxi}:\bm{a}(\bm{x})^{\top}\tilde{\rxi}_i\leq b_i(\bm{x}), \forall i\in [I]\right\} \ge 1-\epsilon. \label{sp-cc}
\end{align}
\end{subequations}
In \eqref{sp}, the vector $\bm x \in \Re^n$ {denotes} the decision {variables}; the vector $\bm c \in \Re^n$ denotes the objective {function} coefficients; the set $S \subseteq \Re^n$ denotes deterministic constraints on $\bm x$; and the constraint (\ref{sp-cc}) is a chance constraint involving $I$ uncertain constraints specified by the random vectors $\tilde{\rxi}_i$ supported on set $\Xi_i\subseteq \Re^{n+1}$ for each $i\in [I]$ with a joint probability distribution $\Pr$ from a family $\P$, termed ``\textit{ambiguity set}''. We let $[R]:=\{1,2,\ldots,R\}$ for any positive integer $R$, and for each uncertain constraint $i \in [I]$, $\bm{a}(\bm{x}) \in \Re^{n+1}$ and $b_i(\bm{x}) \in \Re$ denote affine mappings of $\bm{x}$ such that {$\bm{a}(\bm{x})= \begin{pmatrix}
\eta_1\bm{x}\\
\eta_2
\end{pmatrix}$ and $b_i(\bm{x})=\bm B_i^{\top} \bm x+b^i$ with parameters $\eta_1, \eta_2\in \{0,1\}, \eta_1+\eta_2\geq 1$, $\bm B_i\in \Re^{n}$, and $b^i\in \Re$, respectively. For notational convenience, we let $\Xi\subseteq\prod_{i\in [I]}\Xi_i$ and $\tilde{\rxi}=(\tilde{\rxi}_1,\ldots,\tilde{\rxi}_I)$. Note that (i) for any $i,j\in [I]$ and $i\neq j$, the random vectors $\tilde{\rxi}_{i}$ and $\tilde{\rxi}_j$ can be correlated; and (ii) we use $\eta_1, \eta_2$ to differentiate whether \eqref{sp-cc} involves left-hand uncertainty (i.e., $\eta_1=1,\eta_2=0$), right-hand uncertainty (i.e., $\eta_1=0,\eta_2=1$) or both-side uncertainty (i.e., $\eta_1=1,\eta_2=1$).}

The {distributionally robust chance constraint (DRCC) (\ref{sp-cc})} requires that all $I$ uncertain constraints are simultaneously satisfied for all the probability distributions from ambiguity set $\P$ with a probability at least $(1-\epsilon)$, where $\epsilon\in (0,1)$ is a specified risk tolerance. We call \eqref{sp} a {\em single} DRCCP if $I = 1$ and a {\em joint} DRCCP if $I \ge 2$. Also, \eqref{sp} is termed a DRCCP with {\em right-hand} uncertainty if $\eta_1=0,\eta_2=1$ and a DRCCP with {\em left-hand} uncertainty if $\eta_1=1,\eta_2=0$. For a joint DRCCP, if $I=2, \tilde{\rxi}_1=-\tilde{\rxi}_2$, we call \eqref{sp} as a \emph{two-sided} DRCCP.

We denote the feasible region induced by DRCC \eqref{sp-cc} as
\begin{equation}
\zsetd:=\left\{\bm{x}\in \Re^n:\inf_{\Pr\in \P}\Pr\left\{\tilde{\rxi}:\bm{a}(\bm{x})^{\top}\tilde{\rxi}_i \leq b_i(\bm{x}), \forall i\in [I]\right\} \ge 1-\epsilon\right\}. \label{zset}
\end{equation}

\exclude{
\subsection{Remark about the Setting}
We remark that the form of DRCC (\ref{sp-cc}) is quite general, for example, it covers the setting studied in \cite{hanasusanto2015distributionally}, which is:
\begin{align*}\inf_{\Pr\in \P}\Pr\left\{\hat{\rxi}:\bm{a}_i(\bm{x})^{\top}\hat{\rxi}\leq b_i(\bm{x}), \forall i\in [I]\right\}\ge 1-\epsilon.
\end{align*}
Above, $\bm{a}_i(\bm{x}) =\bm{A}_i\bm{x}+\bm{a}^i$ denotes an affine mapping with $\bm{A}_i\in \Re^{\hat{m}\times n}$ and $\bm{a}^i\in \Re^{\hat{m}}$ for each $i\in [I]$, and the random vector $\hat{\rxi}$ is supported on $\hat{\Xi}\subseteq \Re^{\hat{m}}$. Note that for each $i\in [I]$,
we have $\bm{a}_i(\bm{x})^{\top}\hat{\rxi}=\bm{x}^\top \left({\bm{A}}_i^\top \hat{\rxi}\right)+\left({\bm{a}}^i\right)^\top \hat{\rxi}$. Thus, in DRCC \eqref{sp-cc}, we let $\bm{a}(\bm{x})= \begin{pmatrix}
\eta_1\bm{x}\\
\eta_2
\end{pmatrix}$, where $\eta_1\in \{0,1\}$ indicates whether there exists $i\in [I]$ such that ${\bm{A}}_i\neq \bm{0}$ or not, and $\eta_2\in \{0,1\}$ indicates whether there exists $i\in [I]$ such that ${\bm{a}}^i\neq \bm{0}$ or not. We also define the random vector $\tilde{\bm\xi}_i=\begin{pmatrix}
{\bm{A}}_i^\top \hat{\rxi}\\
\left({\bm{a}}^i\right)^\top \hat{\rxi}
\end{pmatrix}$ for each $i\in [I]$, and their support $\Xi=\left\{\bm{\xi}: \exists \bar{\rxi}\in \hat{\Xi} ,\bm{\xi}_i=\begin{pmatrix}
{\bm{A}}_i^\top \bar{\rxi}\\
\left({\bm{a}}^i\right)^\top \bar{\rxi}
\end{pmatrix},\forall i\in [I]\right\}$. Thus, we arrive at the DRCC studied in \cite{hanasusanto2015distributionally}.
}

{\subsection{Assumptions}}


%

In this paper, we consider Wasserstein ambiguity set $\P$, {i.e., we make the following assumption on the ambiguity set $\P$.}
{\begin{enumerate}
	\setcounter{enumi}{0}
	\renewcommand{\theenumi}{(A\arabic{enumi})}
	\renewcommand{\labelenumi}{\theenumi}
	\item\label{assume_A1} The Wasserstein ambiguity set $\P$ is defined as
	\begin{align}\label{eq_general_das}
\P^{W}=\left\{\Pr:\Pr\left\{\tilde\rxi\in \Xi\right\}=1,W\left(\Pr,\Pr_{\bzeta}\right)\leq \delta\right\},
\end{align}
where Wasserstein distance is defined as
\[W\left(\Pr_1,\Pr_{2}\right)=\inf\left\{\int_{\Xi\times\Xi}\|{\bm{\xi}}_1-{\bm{\xi}}_2\|\Qe(d\bm{\xi}_1,d\bm{\xi}_2):\begin{array}{l}\text{$\Qe$ is a joint distribution of $\hat{\bm{\xi}}_1$ and $\hat{\bm{\xi}}_2$}\\
\text{with marginals $\Pr_1$ and $\Pr_2$, respectively}\end{array}\right\},\]
and $\Pr_{\bzeta}$ denotes a discrete empirical distribution of $\bzeta$ generated by i.i.d. samples $\Z=\{\bm{\zeta}^j\}_{j\in [N]}\subseteq \Xi$ from the true distribution $\Pr^{\infty}$, i.e., its point mass function is $\Pr_{\bzeta}\left\{\bzeta=\bm{\zeta}^j\right\}=\frac{1}{N}$, and $\delta>0$ denotes the Wasserstein radius. We assume that $(\Xi,\|\cdot\|)$ is a totally bounded Polish (separable complete metric) space with distance metric $\|\cdot\|$, i.e., for every $\hat{\epsilon}>0$, there exists a finite covering of $\Xi$ by balls with radius at most $\hat{\epsilon}$.
\end{enumerate}}

Note that the Wasserstein metric measures the distance between true distribution and empirical distribution and is able to recover the true distribution when the number of sampled data goes to infinity \cite{fournier2015rate}. {The fact that the convergence result is not affected by the support motivates us to consider relaxing the support $\Xi=\Re^{I\times (n+1)}$, which provides us better reformulation power.} {That is, we make the following assumption about the support $\Xi$.
\begin{enumerate}
	\setcounter{enumi}{1}
	\renewcommand{\theenumi}{(A\arabic{enumi})}
	\renewcommand{\labelenumi}{\theenumi}
	\item\label{assume_A2} The support $\Xi=\Re^{I\times (n+1)}$, i.e., $\Xi=\prod_{i\in [I]}\Xi_i$ and $\Xi_i=\Re^{n+1}$.
\end{enumerate}
We remark that 
\bl{\begin{itemize}
\item This assumption has been studied in recent DRCCP literature \cite{cheng2014distributionally,xie2017optimized,zhang2016drccbp};
\item By making this assumption, it might cause the DRCC (\ref{sp-cc}) to be more conservative than the general setting studied in \cite{hanasusanto2015distributionally};
\item The interdependence between different random vectors ${\bm{\xi}}_i$ can be inherited implicitly from the empirical distribution. For example, suppose that in the true distribution $\Pr^{\infty}$, we have $\Pr^{\infty}\left\{\rxi_{i_1}=\rxi_{i_2}\right\}=1$ for some $i_1,i_2\in [I]$ and $i_1\neq i_2$, then for any empirical sample $j\in [N]$, we must have $\bm{\zeta}_{i_1}^j= \bm{\zeta}_{i_2}^j$ with probability one. Since the empirical distribution will converge to the true distribution $\Pr^{\infty}$ according to Lemma 3.7 \cite{esfahani2015data} (i.e., when $N\rightarrow \infty$, $\delta\rightarrow 0$), thus Wasserstein Ambiguity set \eqref{eq_general_das} will eventually pick up the fact that $\Pr^{\infty}\left\{\rxi_{i_1}=\rxi_{i_2}\right\}=1$. However, this process might require many more samples than that without Assumption \ref{assume_A2};
\item In practice, one needs to choose a proper Wasserstein radius $\delta$ through cross validation \cite{esfahani2015data} to alleviate
the over-conservatism caused by Assumption \ref{assume_A2}, which will be illustrated in Section~\ref{sec_sep_numerical}. 
\end{itemize}
}

Finally, we suppose that Assumptions~\ref{assume_A1} and \ref{assume_A2} hold throughout the paper.

\subsection{Related Literature}
There are significant \bl{works} on reformulation, convexity and approximations of set $Z$ under various ambiguity sets \cite{calafiore2006distributionally,hanasusanto2015distributionally,hanasusanto2015Ambiguous,jiang2013data,Xie2016drccp,yang2014distributionally}).
For a single DRCCP, when $\P$ consists of all probability distributions with given first and second moments, the set $\zsetd$ is second-order conic representable \cite{calafiore2006distributionally,el2003worst}. Similar convexity results hold for single DRCCP when $\P$ also incorporates other distributional information such as the support of $\tilde{\rxi}$ \cite{cheng2014distributionally}, the unimodality of $\Pr$ \cite{hanasusanto2015distributionally,li2016ambiguous}, or arbitrary convex mapping of $\tilde{\rxi}$ \cite{Xie2016drccp}. 
For a joint DRCCP, \cite{hanasusanto2015Ambiguous} provided the first convex reformulation of $\zsetd$ in the absence of coefficient uncertainty, i.e., $\eta_1 = 0$, when $\P$ is characterized by the mean, a positively homogeneous dispersion measure, and conic support of $\tilde{\rxi}$. For the more general coefficient uncertainty setting, \cite{Xie2016drccp} identified several sufficient conditions for $\zsetd$ to be convex (e.g., when $\P$ is specified by only one moment constraint), and \cite{xie2016opf} showed that $\zsetd$ is convex for two-sided DRCCP when $\P$ is characterized by the first two moments.

When DRCC set $Z$ is not convex, many inner convex approximations have been proposed. 
In \cite{chen2010cvar}, the authors proposed to aggregate the multiple uncertain constraints with positive scalars in to a single constraint, and then use conditional value-at-risk ($\CVaR$) approximation scheme \cite{nemirovski2006convex} to develop an inner approximation of $\zsetd$. This approximation is shown to be exact for single DRCCP when $\P$ is specified by first and second moments in \cite{zymler2013distributionally} or, more generally, by convex moment constraints in~\cite{Xie2016drccp}. In \cite{xie2017optimized}, the authors provided several sufficient conditions under which the well-known Bonferroni approximation of joint DRCCP is exact and yields a convex reformulation.

Recently, there are many successful developments on data-driven distributionally robust programs with Wasserstein ambiguity set \eqref{eq_general_das} \cite{gao2016distributionally,esfahani2015data,zhao2015data_b}. For instance, \cite{gao2016distributionally,esfahani2015data} studied its reformulation under different settings. Later on, \cite{blanchet2016robust,gao2017wasserstein,lee2017minimax,shafieezadeh2015distributionally} applied it to the optimization problems related with machine learning. Other relevant works can be found \cite{blanchet2018distributionally,hanasusanto2016conic,kiesel2016wasserstein,luo2017decomposition}. However, there is very limited literature on DRCCP with Wasserstein ambiguity set.  In \cite{Xie2018approx}, the authors proved that it is strongly NP-hard to optimize over the DRCC set $Z$ with Wasserstein ambiguity set and proposed a bicriteria approximation for a class of DRCCP with covering uncertain constraints (i.e., $S$ is a closed convex cone and $\Xi_i\in \Re_-^n,\bm{B}_i\in \Re_+^n, b_i\in \Re_{-}$ for each $i\in [I]$). In \cite{duan2018distributionally}, the authors considered two-sided DRCCP with right-hand uncertainty and proposed its tractable reformulation, while in \cite{hota2018data}, the authors studied CVaR approximation of DRCCP. \bl{While this paper was under review, we became aware of the independent works \cite{chen2018data,ji2018data}, which developed approximations and exact reformulations DRCCP with Wasserstein ambiguity set. Similar to this paper, in \cite{chen2018data}, the authors also derived exact mixed integer programming reformulations for single DRCCP and DRCCP with right-hand uncertainty using a different proof technique. In \cite{ji2018data}, the authors provided exact reformulations for single DRCCP under discrete support (i.e., $|\Xi|<\infty$) and approximations for single DRCCP under continuous support.}

\subsection{Contributions}
In this paper, we study approximations and exact reformulations of DRCCP under Wasserstein ambiguity set. In particular, our main contributions are summarized as below.
\begin{enumerate}
\item We derive a deterministic equivalent reformulation for set $Z$ and show that this reformulation admits a conditional value-at-risk ($\CVaR$) interpretation, { i.e.,
	\begin{equation*}
	\zsetd=\left\{\bm{x}\in \Re^n:\frac{\delta}{\epsilon}+\CVaR_{1-\epsilon}\left[-f(\bm x,\bzeta)\right]\leq 0\right\}, 
	\end{equation*}
	where $f(\cdot,\cdot)$ is defined in Theorem~\ref{thm_exact_norm}.}
\item {We show that set $Z$, once bounded, is mixed integer representable with big-M coefficients and $N$ additional binary variables.}

\item {We derive inner and outer approximations based upon $\CVaR$ interpretation. We develop compact formulations for these approximations and compare their strengths.}


\item When the decision variables are pure binary (i.e., $S\subseteq \{0,1\}^n$), we first show that the nonlinear constraints in the reformulation can be recast as submodular knapsack constraints. Then, by {exploiting} the polyhedral properties of submodular functions, we propose a new big-M free mixed integer linear reformulation\bl{, which }
can be effectively solved by {a} branch and cut algorithm.
\end{enumerate}


The remainder of the paper is organized as follows. Section \ref{sec_general} presents exact {reformulations} of DRCC set $Z$. {Section \ref{sec_Approximation} provides inner and outer approximations of set $Z$ and compares their strengths.} Section \ref{sec_DRMKP} studies binary DRCCP (i.e., $S\subseteq \{0,1\}^n$), develops a big-M free formulation. {Section~\ref{sec_sep_numerical}} numerically illustrates the proposed methods. Section~\ref{sec_conclusion} concludes the paper.\\

\noindent {\em Notation:} The following notation is used throughout the paper. We use bold-letters (e.g., $\bm{x},\bm{A}$) to denote vectors or matrices, and use corresponding non-bold letters to denote their components. We let $\e$ be the all-ones vector, and let $\e_i$ be the $i$th standard basis vector. 
Given an integer $n$, we let $[n]:=\{1,2,\ldots,n\}$, and use $\Re_+^n:=\{\bm{x}\in \Re^n:x_l\geq0, \forall l\in [n]\}$ and $\Re_-^n:=\{\bm{x}\in \Re^n:x_l\leq0, \forall l\in [n]\}$. Given a real number $t$, we let $(t)_+:=\max\{t,0\}$. Given a finite set $I$, we let $|I|$ denote its cardinality. We let $\tilde{\bm{\xi}}$ denote a random vector with support $\Xi$ and denote one of its realization by $\bm{\xi}$. Given a set $R$, the characteristic function $\chi_{R}(\bm{x})=0$ if $\bm{x}\in R$, and $\infty$, otherwise, while the indicator function $\I(\bm{x}\in R)$ =1 if $\bm{x}\in R$, and 0, otherwise. 
For a matrix $\bm{A}$, we let $\bm{A}_{i\bullet}$ denote $i$th row of $\bm{A}$ and $\bm{A}_{\bullet j}$ denote $j$th column of $\bm{A}$. 
Additional notation will be introduced as needed. Given a subset $T\subseteq [n]$, we define an $n$-dimensional binary vector $\e_T$ as $(\e_T)_{\tau}=\begin{cases}
1,& \text{if }\tau\in T\\
0,& \text{if }\tau\in [n]\setminus T
\end{cases}$.

{
\section{Exact Reformulations}\label{sec_general}

In this section, we will show that DRCC set $\zsetd$ admits a conditional-value-at-risk ($\CVaR$) interpretation and is mixed integer representable. 
This reformulation 
also allows us to derive tight inner and outer approximations {in next section}.

\subsection{$\CVaR$ Reformulation}
In this subsection, we will reformulate the set $\zsetd$ into its deterministic counterpart with respect to empirical distribution. The main idea of this reformulation is that we first use the strong duality result from \cite{blanchet2016quantifying,gao2016distributionally} to formulate the worst-case chance constraint into its dual form, and then break down the indicator function according to its definition.
\begin{theorem}\label{thm_exact_norm}
	Set $\zsetd$ is equivalent to 
	\begin{subequations}\label{zset_P_W_2}
		\begin{empheq}[left={\zsetd=\empheqlbrace{\bm{x}\in \Re^n:}},right=\empheqrbrace]{align}
		& \delta-\epsilon\gamma\leq \frac{1}{N}\sum_{j\in [N]} \min\left\{f(\bm{x},\bm{\zeta}^j)-\gamma,0\right\} ,\label{eq_gamma_delta}\\
		&\gamma\geq 0,
		\end{empheq}
	\end{subequations}
	where\begin{align}
	f(\bm{x},\bm{\zeta})=\min\left\{\min_{i\in [I]\setminus\Ie(\bm x)}\frac{\max\left\{b_i(\bm{x})-\bm{a}(\bm{x})^{\top}\bm{\zeta}_i,0\right\}}{\|\bm{a}(\bm{x})\|_*}, \min_{i\in \Ie(\bm x)}\chi_{\{\bm{x}:b_i(\bm{x})<0\}}(\bm x)\right\},\label{eq_def_f}
	\end{align}
and $\Ie(\bm x)=\emptyset$ if $\bm{a}(\bm{x})\neq0$ and $\Ie(\bm x)=[I]$, otherwise, and characteristic function $\chi_{\R}(\bm x)=\infty$ if $\bm x\notin \R$ and 0, otherwise. 
\end{theorem}
\begin{subequations}
\begin{proof} We separate the proof into three steps, where the first step is to apply strong duality result for distributionally robust optimization, the second step is to break down the indicator function, and the third step is to replace the dual variable with its reciprocal.
\begin{enumerate}[(i)]
\item	Note that
	\[\inf_{\Pr\in \P}\Pr\left\{\tilde{\rxi}:\bm{a}(\bm{x})^{\top}\tilde{\rxi}_i\leq b_i(\bm{x}), \forall i\in [I]\right\}\geq 1-\epsilon\]
	is equivalent to
	\[\sup_{\Pr\in \P}\Pr\left\{\tilde{\rxi}:\bm{a}(\bm{x})^{\top}\tilde{\rxi}_i> b_i(\bm{x}), \exists i\in [I]\right\}\leq \epsilon.\]
	
Since
$$\Pr\left\{\tilde{\rxi}:\bm{a}(\bm{x})^{\top}\tilde{\rxi}_i> b_i(\bm{x}), \exists i\in [I]\right\}=\E_{\Pr}\left[\I\left(\bm{a}(\bm{x})^{\top}\tilde{\rxi}_i> b_i(\bm{x}), \exists i\in [I]\right)\right]$$
and the indicator function is always bounded and upper semi-continuous, therefore, according to Theorem 1 in \cite{gao2016distributionally} or Theorem 1 in \cite{blanchet2016quantifying}, $\sup_{\Pr\in \P}\Pr\left\{\tilde{\rxi}:\bm{a}(\bm{x})^{\top}\tilde{\rxi}_i> b_i(\bm{x}), \exists i\in [I]\right\}$ is equivalent to
	\begin{align}
	& \min_{\lambda \geq 0} \left\{\lambda \delta-\frac{1}{N}\sum_{j\in [N]} \inf_{\rxi}\left[\lambda {\|\rxi-\bm{\zeta}^j\|}-\I\left(\bm{a}(\bm{x})^{\top}\rxi_i> b_i(\bm{x}), \exists i\in [I]\right)\right]\right\}.\label{sp_svm1}
	\end{align}
	
	Thus, set $\zsetd$ becomes
	\begin{equation}
	\zsetd:=\left\{\bm{x}\in \Re^n:\lambda \delta-\frac{1}{N}\sum_{j\in [N]} \inf_{\rxi}\left[\lambda {\|\rxi-\bm{\zeta}^j\|}-\I\left(\bm{a}(\bm{x})^{\top}\rxi_i> b_i(\bm{x}), \exists i\in [I]\right)\right]\leq \epsilon, \exists\lambda\geq 0\right\}. \label{zset_rhs}
	\end{equation}
	
\item Next, we break down the indicator function in the infimum of \eqref{zset_rhs} \bl{by discussing the conditions under which it is equal to zero or one }and reformulate it as below.\\
	\begin{claim1}\bl{For} given $\lambda \geq 0$ and $\bm{\zeta}\in \Z$, we have
		\begin{align}
		& \inf_{\rxi}\left[ \lambda {\|\rxi-\bm{\zeta}\|}-\I\left(\bm{a}(\bm{x})^{\top}\rxi_i> b_i(\bm{x}), \exists i\in [I]\right)\right]=\min\left\{\min_{i\in [I]}\inf_{ \bm{a}(\bm{x})^{\top}\rxi_i> b_i(\bm{x})}\left[ \lambda {\|\rxi-\bm{\zeta}\|}-1\right],0\right\}.\label{sp_svm_p2}
		\end{align}
	\end{claim1}
	\begin{proof}We first note that $\I\left(\bm{a}(\bm{x})^{\top}\rxi_i> b_i(\bm{x}), \exists i\in [I]\right)=\max_{i\in [I]}\I\left(\bm{a}(\bm{x})^{\top}\rxi_i> b_i(\bm{x})\right)$. Thus,
		\begin{align*}
		& \inf_{\rxi}\left[ \lambda {\|\rxi-\bm{\zeta}\|}-\I\left(\bm{a}(\bm{x})^{\top}\rxi_i> b_i(\bm{x}), \exists i\in [I]\right)\right]=\min_{i\in [I]}\inf_{\rxi}\left[ \lambda {\|\rxi-\bm{\zeta}\|}-\I\left(\bm{a}(\bm{x})^{\top}\rxi_i> b_i(\bm{x})\right)\right].
		\end{align*}
		
		Therefore, we only need to show that for any $i\in [I]$,
		\begin{align}
		& \inf_{\rxi}\left[ \lambda {\|\rxi-\bm{\zeta}\|}-\I\left(\bm{a}(\bm{x})^{\top}\rxi_i> b_i(\bm{x})\right)\right]=\min\left\{\inf_{\bm{a}(\bm{x})^{\top}\rxi_i> b_i(\bm{x})}\left[ \lambda {\|\rxi-\bm{\zeta}\|}-1\right],0\right\}.\label{sp_svm_p3}
		\end{align}	
		
		There are two cases:
		\begin{enumerate}[{Case} 1.]
			\item If $\bm{a}(\bm{x})^{\top}\bm{\zeta}_i> b_i(\bm{x})$, then in the left-hand side of \eqref{sp_svm_p3}, the infimum is equal to $-1$ by letting $\rxi:=\bm{\zeta}$, which equals the right-hand side since the infimum is also achieved by $\rxi:=\bm{\zeta}$.
			\item If $\bm{a}(\bm{x})^{\top}\bm{\zeta}_i\leq b_i(\bm{x})$, then for any $\rxi\in \Xi$, we either have $\bm{a}(\bm{x})^{\top}\rxi_i> b_i(\bm{x})$ or $\bm{a}(\bm{x})^{\top}\rxi_i\leq b_i(\bm{x})$. Hence, the left-hand side of \eqref{sp_svm_p3} is equivalent to
			\begin{align*}
			& \inf_{\rxi}\left[ \lambda {\|\rxi-\bm{\zeta}\|}-\I\left(\bm{a}(\bm{x})^{\top}\rxi_i> b_i(\bm{x})\right)\right]\\
=&\min\left\{\inf_{\bm{a}(\bm{x})^{\top}\rxi_i> b_i(\bm{x})}\left[ \lambda {\|\rxi-\bm{\zeta}\|}-1\right],\inf_{\bm{a}(\bm{x})^{\top}\rxi_i\leq b_i(\bm{x})}\left[\lambda {\|\rxi-\bm{\zeta}\|}\right]\right\}\\
			=&\min\left\{\inf_{\bm{a}(\bm{x})^{\top}\rxi_i> b_i(\bm{x})}\left[ \lambda {\|\rxi-\bm{\zeta}\|}-1\right],0\right\},
			\end{align*}
			where $\inf_{\bm{a}(\bm{x})^{\top}\rxi_i\leq b_i(\bm{x})}\left[ {\|\rxi-\bm{\zeta}\|}\right]=0$ by letting $\rxi:=\bm{\zeta}$.
		\end{enumerate}
		\qedA
	\end{proof}
	According to Claim 1 and the fact that
\[\inf_{ \bm{a}(\bm{x})^{\top}\rxi_i> b_i(\bm{x})}\|\rxi-\bm{\zeta}\|=\begin{cases}
\frac{\max\left\{b_i(\bm{x})-\bm{a}(\bm{x})^{\top}\bm{\zeta}_i,0\right\}}{\|\bm{a}(\bm{x})\|_*}&\text{ if }\bm{a}(\bm{x})\neq \bm{0}\\
\chi_{\{\bm{x}:b_i(\bm{x})<0\}}(\bm x)&\text{ otherwise}\\
\end{cases},\]
set $Z$ becomes
			\begin{equation}
	\zsetd=\left\{\bm{x}\in \Re^n:\lambda\delta-\epsilon\leq \frac{1}{N}\sum_{j\in [N]} \min\left\{\lambda f(\bm{x},\bm{\zeta}^j)-1,0\right\},\lambda\geq 0\right\} .\label{eq_gamma_delta3}
		\end{equation}

\item Finally, let $\zsetd'$ denote the set in the right-hand side of \eqref{zset_P_W_2} , we only need to show that 
$\zsetd=\zsetd'$.
\begin{enumerate}
\item[$(\zsetd\subseteq\zsetd')$] Given $\bm{x}\in \zsetd$, there exists $\lambda\geq 0$ such that $(\bm{x},\lambda)$ satisfies \eqref{zset_P_W_2}. If $\lambda> 0$, then let $\gamma=\frac{1}{\lambda}$. Then it is easy to see that $(\bm{x},\gamma)$ satisfies \eqref{zset_P_W_2} . Hence, $\bm{x}\in \zsetd'$.

Now suppose that $\lambda=0$, then in \eqref{zset_P_W_2}, we have
\[-\epsilon\leq -1\]
a contradiction that $\epsilon \in (0,1)$.

\item[$(\zsetd\supseteq\zsetd')$] Similarly, given $\bm{x}\in \zsetd'$, there exists $\gamma\geq 0$ such that $(\bm{x},\gamma)$ satisfies \eqref{zset_P_W_2} . If $\gamma> 0$, then let $\lambda=\frac{1}{\gamma}$. Then it is easy to see that $(\bm{x},\lambda)$ satisfies \eqref{zset_P_W_2}. Hence, $\bm{x}\in \zsetd$.

Now suppose that $\gamma=0$, then in \eqref{zset_P_W_2} , we have
\[ \min\left\{f(\bm{x},\bm{\zeta}^j)-\gamma,0\right\}:=0\]
for each $j\in [N]$. Thus, \eqref{zset_P_W_2}  reduces to $\delta\leq 0$
contradicting that $\delta>0$.\qed
\end{enumerate}
\end{enumerate}

\end{proof}

\end{subequations}
Please note that in the proof, we use the fact that $\delta>0$ from Assumption~\ref{assume_A1}, and the formulation \eqref{zset_P_W_2} does not hold if $\delta=0$.

An interesting corollary of Theorem~\ref{thm_exact_norm} is that set $Z$ can be reformulated as a conditional-value-at-risk ($\CVaR$) constrained set. Before showing this interpretation, let us first introduce the following two definitions. Given a random variable $\tilde{X}$, let $\Pr$ and $F_{\tilde{X}}(\cdot)$ be its probability distribution and cumulative distribution function, respectively.
Then $(1-\epsilon)$-value at risk (\VaR) of $\tilde{X}$ is
\[\VaR_{1-\epsilon}(\tilde{X}):=\min\left\{s: F_{\tilde{X}}(s)\geq 1-\epsilon\right\},\]
while its $(1-\epsilon)$-conditional value-at-risk (\CVaR) \cite{rockafellar2000optimization} is defined as 
\[\CVaR_{1-\epsilon}(\tilde{X}):=\min_{\beta}\left\{\beta+\frac{1}{\epsilon}\E_{\Pr}\left[\tilde{X}-\beta\right]_+\right\}.\]
With the definitions above, we observe that set $Z$ in \eqref{zset_P_W_2} has a $\CVaR$ interpretation. 
\begin{corollary}\label{ref_math_thm_cor}Set $\zsetd$ is equivalent to 
	\begin{equation}
	\zsetd=\left\{\bm{x}\in \Re^n:\frac{\delta}{\epsilon}+\CVaR_{1-\epsilon}\left[-f(\bm x,\bzeta)\right]\leq 0\right\}, \label{zset_rhs_new2}
	\end{equation}
	where $f(\cdot,\cdot)$ is defined in \eqref{eq_def_f},
	and $\CVaR_{1-\epsilon}\left[-f(\bm x,\bzeta)\right]=\min_{\gamma}\left\{\gamma+\frac{1}{\epsilon}\E_{\Pr_{\bzeta}}\left[-f(\bm x,\bzeta)-\gamma\right]_+\right\}$.
\end{corollary}
\begin{proof}First, we observe that the constraint in \eqref{zset_P_W_2} directly implies $\gamma\geq 0$, thus the nonnegativity constraint of $\gamma$ can be dropped, i.e., equivalently, we have
	\begin{equation*}
	\zsetd=\left\{\bm{x}\in \Re^n:\frac{\delta}{\epsilon}-\gamma+\frac{1}{N\epsilon}\sum_{j\in [N]} \max\left\{-f(\bm{x},\bm{\zeta}^j)+\gamma,0\right\}\leq 0\right\}.
	\end{equation*}	
	Next, in the above formulation, letting \bl{$\gamma':=-\gamma$ and replacing the existence of $\gamma'$ by finding the best $\gamma'$ such that the constraint still holds, we arrive at
		\begin{equation*}
	\zsetd=\left\{\bm{x}\in \Re^n:\frac{\delta}{\epsilon}+\min_{\gamma'}\left\{\gamma'+\frac{1}{N\epsilon}\sum_{j\in [N]} \max\left\{-f(\bm{x},\bm{\zeta}^j)-\gamma',0\right\}\right\} \leq 0\right\},
	\end{equation*}	}
	which is equivalent to \eqref{zset_rhs_new2}.\qed
\end{proof}

In the following sections, we will derive the inner and outer approximations mainly based upon $\CVaR$ formulation in Corollary~\ref{ref_math_thm_cor}.

\subsection{Exact Mixed Integer Program Reformulation}
In this subsection, we show that set $\zsetd$ is mixed integer representable. To do so, we first observe that the reformulation of set $\zsetd$ in Theorem \ref{thm_exact_norm} can be further simplified as a disjunction of a nonconvex set and a convex set.
\begin{proposition}\label{thm_reform_joint_drccp} 
Set $\zsetd=\zsetd_{1}\cup\zsetd_{2}$, where
	\begin{subequations}\label{zset_P_W_I_1-M}
		\begin{empheq}[left={\zsetd_{1}=\empheqlbrace{\bm{x}\in \Re^n:}},right=\empheqrbrace]{align}
		& \delta \nu-\epsilon\gamma \leq \frac{1}{N}\sum_{j\in [N]}z_j,\label{z_v_I_1-M}\\
		&z_j+\gamma\leq \max\left\{b_i(\bm{x})-\bm{a}(\bm{x})^{\top}\bm{\zeta}_i^j,0\right\}, \forall i\in [I],j\in [N],\label{z_lambda_a_b_I_1-M}\\
		&z_j\leq 0, \forall j\in [N],\label{z_0_I_1-M}\\
		&\|\bm{a}(\bm{x})\|_* \leq \nu,\label{a_v_I_1-M}\\
		&\nu>0,\gamma\geq 0,\label{a_v_I_1-M_nu}
		\end{empheq}
	\end{subequations}
	and
	\begin{empheq}[left={\zsetd_{2}=\empheqlbrace{\bm{x}\in \Re^n:}},right=\empheqrbrace.]{align}
	\bm{a}(\bm{x})=\bm0,b_i(\bm{x})\geq 0,\forall i\in [I]\label{zset_P_W_I_2-M}
	\end{empheq}
\end{proposition}
\begin{proof}
We need to show that $Z_1\cup Z_2\subseteq Z$ and $Z\subseteq \zsetd_1\cup\zsetd_2$.
\begin{enumerate}
	\item[$Z_1\cup Z_2\subseteq Z$.]

Given $\bm{x}\in Z_{2}$, we have $\Ie(\bm x)=[I]$, thus $f(\bm{x},\bm{\zeta})$ (defined in \eqref{eq_def_f}) is $\infty$. Thus, let $\gamma=\frac{\delta}{\epsilon}$. Clearly, $(\gamma,\bm x)$ satisfies all the constraints in \eqref{zset_P_W_2}, i.e., $\bm{x}\in Z$. Hence, $Z_{2}\subseteq Z$.
		
Given $\bm{x}\in Z_{1}$, there exists $(\gamma,\nu,\bm z,\bm{x})$ which satisfies constraints in \eqref{zset_P_W_I_1-M}. Suppose that $\Ie(\bm x)=[I]$, then we have $\bm{a}(\bm{x})=\bm0$. Hence, for each $i\in \Ie(\bm x)$, we have \eqref{z_v_I_1-M} and \eqref{z_lambda_a_b_I_1-M} imply that
		\[\delta \nu-\epsilon\gamma \leq \frac{1}{N}\sum_{j\in [N]}z_j \leq\frac{1}{N}\sum_{j\in [N]}(\max\left\{b_i(\bm{x}),0\right\}-\gamma),\]
		which is equivalent to 
		$$\max\left\{b_i(\bm{x}),0\right\} \geq \delta \nu+(1-\epsilon)\gamma>0.$$
		That is, $b_i(\bm{x})>0$.  Thus, $\bm{x}\in Z_2\subseteq Z$.

Now we suppose that $\Ie(\bm x)=\emptyset$. For each $i\in [I]$, \eqref{z_v_I_1-M} and \eqref{z_lambda_a_b_I_1-M} along with $\nu>0$ imply that
\begin{align*}
\frac{z_j}{\nu}&\leq \min\left\{\frac{1}{\nu}\min_{i\in [I]}\max\left\{b_i(\bm{x})-\bm{a}(\bm{x})^{\top}\bm{\zeta}_i^j,0\right\}-\frac{\gamma}{\nu},0\right\}\\
&\leq \min\left\{\min_{i\in [I]}\frac{\max\left\{b_i(\bm{x})-\bm{a}(\bm{x})^{\top}\bm{\zeta}_i^j,0\right\}}{\|\bm{a}(\bm{x})\|_*}-\frac{\gamma}{\nu},0\right\}\\
&=\min\left\{f(\bm{x},\bm{\zeta}^j)-\frac{\gamma}{\nu},0\right\}
\end{align*}
where the second inequality \bl{is} due to \eqref{a_v_I_1-M}. Then according to \eqref{z_v_I_1-M}, we have
\begin{align*}
\delta -\epsilon\frac{\gamma}{\nu} \leq \frac{1}{N}\sum_{j\in [N]}\frac{z_j}{\nu}\leq \frac{1}{N}\sum_{j\in [N]}\min\left\{f(\bm{x},\bm{\zeta}^j)-\frac{\gamma}{\nu},0\right\}
\end{align*}
i.e., $(\gamma/\nu,\bm{x})$ satisfies the constraints in \eqref{zset_P_W_2}, i.e., $\bm{x}\in Z$. 
Thus, $ Z_{1}\subseteq Z$. 	

\item[$Z\subseteq \zsetd_1\cup\zsetd_2$.] Similarly, given $\bm{x}\in Z$, there exists $(\gamma,\bm{x})$ which satisfies constraints in \eqref{zset_P_W_2}. Suppose that $\bm{a}(\bm{x})=\bm0$, then we must have $b_i(\bm{x})\geq 0$ for all $i\in [I]$, otherwise, we have $f(\bm{x},\bm{\zeta}^j)=0$ for all $j\in [I]$. Then \eqref{eq_gamma_delta} is equivalent to
\[0<\delta\leq (\epsilon-1)\gamma\]
a contradiction that $\gamma\geq 0, \epsilon \in (0,1)$. Hence, we must $\bm{x}\in Z_2$.

From now on, we assume that $\bm{a}(\bm{x})\neq 0$. Let us define $\hat{\gamma}=\gamma \|\bm{a}(\bm{x})\|_*, \nu=\|\bm{a}(\bm{x})\|_*$, and $z_j=\min_{i\in [I]}(\max\{b_i(\bm{x})-\bm{a}(\bm{x})^{\top}\bm{\zeta}_i^j,0\}-\hat{\gamma},0)$ for each $j\in [N]$. Clearly, $(\hat{\gamma},\nu,\bm z,\bm{x})$ satisfies constraints in \eqref{zset_P_W_I_1-M}, i.e., $\bm{x}\in Z_1$. \qed
\end{enumerate}	
\end{proof}

We make the following remarks about the disjunctive formulation of set $Z$.
\begin{remark}
\begin{enumerate}[(i)]
\item Set $Z_2$ is trivial:
\begin{itemize}
\item For DRCCP with left-hand uncertainty (i.e., $\eta_1=1,\eta_2=0$), we have
\begin{align*}
Z_2=\left\{\bm{x}\in \Re^n: \bm{x}=0,b_i\geq 0,\forall i\in [I]\right\};
\end{align*} 
\item For DRCCP with right-hand uncertainty or two-side uncertainty (i.e., $\eta_1\in \{0,1\},\eta_2=1$), we have $Z_2=\emptyset$.
\end{itemize}
\item According to Lemma 2 \cite{Xie2016drccp}, the feasible region induced by a chance constraint is closed, so is set $Z$. However, set $Z_1$ might not be closed due to $\nu>0$ in \eqref{a_v_I_1-M_nu}. In practice, one can find a lower bound $0<\underline{\nu}$ such that 
$$\underline{\nu} \leq \inf_{\bm{x}\in Z_1}\left\{\|\bm{a}(\bm{x})\|_*: \|\bm{a}(\bm{x})\|_*\neq 0\right\};$$
or let $\underline{\nu}$ be a sufficiently small number.
Then replace the constraint $\nu>0$ in \eqref{a_v_I_1-M_nu} by $\nu\geq \underline{\nu}$.
\end{enumerate}
\end{remark}

We observe that set $Z_1$ can be formulated as a mixed integer set when it is bounded, i.e., we can use binary variables to represent the nonlinear constraints \eqref{z_lambda_a_b_I_1-M} as mixed integer linear ones. \bl{This result has been observed independently by \cite{chen2018data} (see their Proposition 1) for single DRCCP.}
\begin{theorem}\label{cor_reform_multi_drccp}
	Suppose there exists an $\bm M\in \Re_+^{N}$ such that 
	$$\max_{i\in [I]}\max_{\bm{x} \in Z_1}\left\{\big|b_i(\bm{x})-\bm{a}(\bm{x})^{\top}\bm{\zeta}_i^j\big|\right\}\leq M_{j}$$
	for all $j\in [N]$. Then $ \zsetd_1$ is mixed integer representable, i.e.,
	\begin{subequations}\label{zset_P_W_I_1+_bigM}
		\begin{empheq}[left={\zsetd_1=\empheqlbrace{\bm{x} \in \Re^n:}},right=\empheqrbrace]{align}
		& \delta \nu-\epsilon\gamma \leq \frac{1}{N}\sum_{j\in [N]}z_j,\\
		& z_j+\gamma\leq s_{j}, \forall j\in [N],\\
		&s_{j}\leq b_i(\bm{x})-\bm{a}(\bm{x})^{\top}\bm{\zeta}_i^j +M_{j}(1-y_{j}),\forall i\in [I], j\in [N],\\
		&s_{j}\leq M_{j} y_{j}, \forall j\in [N],\\
		&\|\bm{a}(\bm{x})\|_* \leq \nu,\\
		&\nu>0,\gamma\geq 0,s_{j}\geq 0,z_j\leq 0, y_{j}\in \{0,1\} ,\forall j\in [N].
		\end{empheq}
	\end{subequations}
\end{theorem}
\begin{proof}
We first observe that the constraints \eqref{z_lambda_a_b_I_1-M} are equivalent to
	\[z_j+\gamma\leq \max\left\{\min_{i\in [I]}b_i(\bm{x})-\bm{a}(\bm{x})^{\top}\bm{\zeta}_i^j,0\right\}, \forall j\in [N].\]
Above, the outer maximum in the right-hand side can be linearized by using a binary variable $y_j$, a continuous variable $s_j$, and big-M coefficient $M_j$ for each $j\in [N]$. By doing so, we arrive at \eqref{zset_P_W_I_1+_bigM}.
\qed
\end{proof}
\bl{Usually, we can derive the big-M coefficients by inspection; for example, suppose that $\bm{x}\in [\bm{L},\bm{U}]$, then for each $j\in [N]$, we can find $M_j$ in the following way: (i) rewrite $b_i(\bm{x})-\bm{a}(\bm{x})^{\top}\bm{\zeta}_i^j=\sum_{\tau\in [n]}(B_{i\tau}-\eta_1 \zeta_{i\tau}^j)x_\tau+b^i-\eta_2\zeta_{i(n+1)}^j$ for each $i\in [I]$, (ii) define sets $\hat{S}_+=\{\tau\in [n]: B_{i\tau}-\eta_1 \zeta_{i\tau}^j>0\}$ and $\hat{S}_-=[n]\setminus\hat{S}_+$, and (iii) let $M_j$ be
\begin{align*}
M_j:=\max_{i\in [I]}\max&\left\{\sum_{\tau\in \hat{S}_+}(B_{i\tau}-\eta_1 \zeta_{i\tau}^j)U_\tau+\sum_{\tau\in \hat{S}_-}(B_{i\tau}-\eta_1 \zeta_{i\tau}^j)L_\tau+b^i-\eta_2\zeta_{i(n+1)}^j,\right.\\
&\left.-\sum_{\tau\in \hat{S}_+}(B_{i\tau}-\eta_1 \zeta_{i\tau}^j)L_\tau-\sum_{\tau\in \hat{S}_-}(B_{i\tau}-\eta_1 \zeta_{i\tau}^j)U_\tau-b^i+\eta_2\zeta_{i(n+1)}^j\right\}.
\end{align*}
There are various methods introduced in literature \cite{qiu2014covering,song14chance} to further tighten big-M coefficients.}

Formulation \eqref{zset_P_W_I_1+_bigM} involves $N$ binary variables and big-M coefficients. In Section \ref{sec_DRMKP}, we will show that for binary DRCCP, set $Z_1$ can be reformulated as a big-M free formulation without introducing additional binary variables.

\subsection{A Special Case: DRCCP with Right-hand Uncertainty}

In this subsection, we consider DRCCP with right-hand uncertainty, i.e., $\eta_1=0,\eta_2=1,\bm{a}(\bm{x})=\e_{n+1}$. We first observe that when $\bm{a}(\bm{x})=\e_{n+1}\neq \bm{0}$, in Theorem \ref{thm_exact_norm}, set $\zsetd$ of DRCCP with right-hand uncertainty has a more compact representation.
\begin{corollary} For DRCCP with \bl{right-hand uncertainty} (i.e., $\eta_1=0,\eta_2=1,\bm{a}(\bm{x})=\e_{n+1}$), set $\zsetd$ is equivalent to the following mathematical program:
	\begin{subequations}\label{zset_P_W_rhs}
		\begin{empheq}[left={\zsetd=\empheqlbrace{\bm{x}\in \Re^n:}},right=\empheqrbrace]{align}
		& \delta\|\e_{n+1}\|_*-\epsilon\gamma \leq \frac{1}{N}\sum_{j\in [N]}z_j,\label{eq_rhs_delta_eps_max}\\
		& z_j+\gamma\leq \max\left\{ b_i(\bm{x})-\e_{n+1}^{\top}\bm{\zeta}_i^j,0\right\}, \forall j\in [N], i\in [I],\label{eq_rhs_z-x_gamma_max}\\
		&z_j\leq 0, \forall j\in [N],\gamma\geq 0.\label{eq_rhs_bnd_gamma_max}
		\end{empheq}
	\end{subequations}
\end{corollary}
\begin{proof}
The result directly follows from Theorem \ref{thm_exact_norm}.\qed
\end{proof}
The differences between this result and the one in Proposition~\ref{thm_reform_joint_drccp} are: (i) for DRCCP with right-hand uncertainty, we do not need to reformulate set $Z$ as a disjunction of two sets, and (ii) compared to set $Z_1$ in \eqref{zset_P_W_I_1-M}, there is no need to introduce additional positive variable $\nu$ in the formulation \eqref{zset_P_W_rhs}.

Following the similar derivation in Theorem~\ref{cor_reform_multi_drccp}, we can also reformulate the set $Z$ in \eqref{zset_P_W_rhs} as a mixed integer program as below. \bl{This result has been observed independently by  \cite{chen2018data} (see their Proposition 2).}
\begin{corollary}
	For DRCCP with \bl{right-hand uncertainty} (i.e., $\eta_1=0,\eta_2=1,\bm{a}(\bm{x})=\e_{n+1}$), suppose that there exists an $\bm{M}\in \Re_+^{N}$ such that 
	$$\max_{i\in [I]}\max_{\bm{x} \in Z}\left\{|b_i(\bm{x})-\e_{n+1}^{\top}\bm{\zeta}_i^j|\right\}\leq M_{j}$$
	for all $j\in [N],i\in [I]$. Then set $ \zsetd$ is mixed integer representable, i.e.,
	\begin{subequations}\label{zset_P_W__RHS_bigM}
		\begin{empheq}[left={\zsetd=\empheqlbrace{\bm{x} \in \Re^n:}},right=\empheqrbrace]{align}
		& \delta\|\e_{n+1}\|_*-\epsilon\gamma \leq \frac{1}{N}\sum_{j\in [N]}z_j,\label{eq_rhs_delta_eps}\\
		&z_j+\gamma\leq  s_{j} , \forall j\in [N], \label{eq_rhs_z-s_gamma}\\
		&s_{j}\leq b_i(\bm{x})-\e_{n+1}^{\top}\bm{\zeta}_i^j+M_j(1-y_{j}), \bl{\forall i\in [I],}\forall j\in [N],\label{eq_rhs_s-y_gamma}\\
		&s_{j}\leq M_{j}y_{j}, \forall j\in [N], \label{eq_rhs_big_M}\\
		&\gamma\geq 0,z_j\leq 0, s_j\geq 0,y_{j}\in \{0,1\} ,\forall j\in [N].\label{eq_rhs_bnd_gamma}
		\end{empheq}
	\end{subequations}
\end{corollary}
\begin{proof}The proof is similar as that of Theorem~\ref{cor_reform_multi_drccp}, thus is omitted.
\qed
\end{proof}
\bl{Similar to Theorem~\ref{cor_reform_multi_drccp}, suppose that $\bm{x}\in [\bm{L},\bm{U}]$, then for each $j\in [N]$, one possible $M_j$ can be derived as below:
\begin{align*}
M_j:=\max_{i\in [I]}\max&\left\{\sum_{\tau\in \hat{S}_+}B_{i\tau}U_\tau+\sum_{\tau\in \hat{S}_-}B_{i\tau}L_\tau+b^i-\zeta_{i(n+1)}^j,-\sum_{\tau\in \hat{S}_+}B_{i\tau}L_\tau-\sum_{\tau\in \hat{S}_-}B_{i\tau}U_\tau-b^i+\zeta_{i(n+1)}^j\right\},
\end{align*}
where $\hat{S}_+=\{\tau\in [n]: B_{i\tau}>0\}$ and $\hat{S}_-=[n]\setminus\hat{S}_+$.}


\section{Outer and Inner Approximations}\label{sec_Approximation}

In this section, we will introduce one outer approximation and three different inner approximations by exploiting the exact reformulations in the previous section. \bl{The outer approximation can provide a lower bound for DRCCP, while inner approximations can provide good-quality feasible solutions. Our numerical study in Section \ref{sec_sep_numerical} will demonstrate that together these approximations, we can obtain better solutions than those from the exact mixed integer programming model in the previous section, in particular, for large-sized instances.}

\subsection{$\VaR$ Outer Approximation} 

Note from \cite{rockafellar2000optimization} that for any random variable $\tilde{X}$, we have
$$\CVaR_{1-\epsilon}\left(\tilde{X}\right)=\VaR_{1-\epsilon}\left(\tilde{X}\right)+\frac{1}{\epsilon}\E\left[\tilde{X}-\VaR_{1-\epsilon}\left(\tilde{X}\right)\right]_+ \geq \VaR_{1-\epsilon}\left(\tilde{X}\right).$$ Therefore, in Corollary~\ref{ref_math_thm_cor}, if we replace $\CVaR_{1-\epsilon}\left(\cdot\right)$ by $\VaR_{1-\epsilon}\left(\cdot\right)$, then we have the following outer approximation of set $\zsetd$. 
\begin{theorem}\label{cor_var}Set $\zsetd$ can be outer approximated by
	\begin{equation}
	\zsetd_{\VaR}=\left\{\bm{x}\in \Re^n:\Pr_{\bzeta}\left\{\frac{\delta}{\epsilon}\|\bm{a}(\bm{x})\|_*+\bm{a}(\bm{x})^{\top}\bzeta_i\leq b_i(\bm{x}), i\in [I]\right\}\geq 1-\epsilon\right\}. \label{zset_rhs_new2_var2}
	\end{equation}
\end{theorem}
\begin{proof}
	According to Corollary~\ref{ref_math_thm_cor} and the well-known result in \cite{rockafellar2000optimization} that $$\CVaR_{1-\epsilon}\left[-f(\bm x,\bzeta)\right] \geq \VaR_{1-\epsilon}\left[-f(\bm x,\bzeta)\right],$$ set $\zsetd$ can be outer approximated by
	\begin{equation*}
	\zsetd_{\VaR}=\left\{\bm{x}\in \Re^n:\frac{\delta}{\epsilon}+\VaR_{1-\epsilon}\left[-f(\bm x,\bzeta)\right]\leq 0\right\}.
	\end{equation*}

Note that \[f(\bm{x},\bm{\zeta})=\min\left\{\min_{i\in [I]\setminus\Ie(\bm x)}\frac{\max\left\{b_i(\bm{x})-\bm{a}(\bm{x})^{\top}\bm{\zeta},0\right\}}{\|\bm{a}(\bm{x})\|_*}, \min_{i\in \Ie(\bm x)}\chi_{b_i(\bm{x})<0}(x)\right\},\]
and $\Ie(\bm x)=\emptyset$ if $\bm{a}(\bm{x})\neq0$, otherwise, $\Ie(\bm x)=[I]$. Thus we further have
\begin{align*}
		\zsetd_{\VaR}=\left\{\bm{x}\in \Re^n:\Pr_{\bzeta}\left\{\begin{array}{c}
		\frac{\max\left\{	b_i(\bm{x})-\bm{a}(\bm{x})^{\top}\bm{\zeta},0\right\}}{\|\bm{a}(\bm{x})\|_*}\geq \frac{\delta}{\epsilon}, \forall i\in [I]\setminus \Ie(\bm x),\\
		 \chi_{b_i(\bm{x})<0}(x)\geq \frac{\delta}{\epsilon},\forall i\in \Ie(\bm x)
		 \end{array}\right\}\geq 1-\epsilon\right\}.
	\end{align*}

Using the fact that $\frac{\delta}{\epsilon}>0$, we arrive at \eqref{zset_rhs_new2_var2}.\qed
\end{proof}

We make the following remarks about outer approximation $Z_{\VaR}$.
\begin{enumerate}[(i)]
\item In \eqref{zset_rhs_new2_var2}, we arrive at a regular chance constrained program with discrete random vector $\bzeta$, which can be reformulated as mixed integer program with big-M coefficients (cf., \cite{ahmed2014nonanticipative,luedtke2008sample});
\item A particular interpretation of formulation \eqref{zset_rhs_new2_var2} is that in order to enforce the robustness, we further penalize the left-hand side of uncertain constraints by the dual norm $\|\bm{a}(\bm{x})\|_*$; and
\item Suppose that the empirical distribution will converge to the true distribution $\Pr^{\infty}$ (cf.,  Lemma 3.7 \cite{esfahani2015data}), i.e., $\delta\rightarrow 0$ as $N\rightarrow \infty$. Then $Z_{\VaR}\rightarrow Z$ as $N\rightarrow \infty$. 
\end{enumerate}
This final remark is summarized below.
\begin{proposition}Suppose that the empirical distribution $\Pr_{\bzeta}$ will converge to the true distribution $\Pr^{\infty}$. Then with probability one, we have
$Z_{\VaR}\rightarrow Z$ as $N\rightarrow\infty$.
\end{proposition}

{Recently, there are several works \cite{bertsimas2018data,bertsimas2019twostage,gao2017distributionally,xie2019tractable} on distributionally robust optimization with $\infty-$Wasserstein ambiguity set, and set $Z_{\VaR}$ is in fact equal to the feasible region induced by DRCC with $\infty-$Wasserstein ambiguity set.
\begin{proposition}Consider $\infty-$Wasserstein ambiguity set $\P^{W}$ defined as
	\begin{align}\label{eq_general_das1}
\P_\infty^{W}=\left\{\Pr:\Pr\left\{\tilde\rxi\in \Xi\right\}=1,W_{
\infty}\left(\Pr,\Pr_{\bzeta}\right)\leq \frac{\delta}{\epsilon}\right\},
\end{align}
where $\infty-$Wasserstein distance is defined as
\[W_\infty\left(\Pr_1,\Pr_{2}\right)=\inf_{\Qe}\left\{\text{ess.sup}\|{\bm{\xi}}_1-{\bm{\xi}}_2\|:\begin{array}{l}\text{$\Qe$ is a joint distribution of $\hat{\bm{\xi}}_1$ and $\hat{\bm{\xi}}_2$}\\
\text{with marginals $\Pr_1$ and $\Pr_2$, respectively}\end{array}\right\}.\]
Then set $Z_{\VaR}$ is equivalent to
\[Z_{\VaR}:=\left\{\bm{x}:\inf_{\Pr\in \P^W_\infty}\Pr\left\{\tilde{\rxi}:\bm{a}(\bm{x})^{\top}\tilde{\rxi}_i\leq b_i(\bm{x}), \forall i\in [I]\right\} \ge 1-\epsilon\right\}.\]
\end{proposition}
\begin{proof}
Let us denote set $\hat{Z}:=\left\{\bm{x}:\inf_{\Pr\in \P^W_\infty}\Pr\left\{\tilde{\rxi}:\bm{a}(\bm{x})^{\top}\tilde{\rxi}_i\leq b_i(\bm{x}), \forall i\in [I]\right\} \ge 1-\epsilon\right\}$. Let $1$ minus both sides of the inequalities, and we have
Note that
\begin{align*}
\hat{Z}
=&\left\{\bm{x}:\sup_{\Pr\in \P^W_\infty}\E_{\Pr}\left[\I\left\{\bm{a}(\bm{x})^{\top}\tilde{\rxi}_i>b_i(\bm{x}), \textrm{ for some }i\in [I]\right\} \right]\leq \epsilon \right\}.
\end{align*}
According to Theorem 5 in \cite{bertsimas2018data}, $\sup_{\Pr\in \P^W_\infty}\E_{\Pr}\left[\I\left\{\tilde{\rxi}:\bm{a}(\bm{x})^{\top}\tilde{\rxi}_i>b_i(\bm{x}), \textrm{ for some }i\in [I]\right\} \right]$ is equivalent to
	\begin{align*}
	&\frac{1}{N}\sum_{j\in [N]}\sup_{\bm{\xi}:\|\bm{\xi}-\bzeta^j\|\leq \frac{\delta}{\epsilon}}\I\left\{\bm{a}(\bm{x})^{\top}{\rxi}_i>b_i(\bm{x}), \textrm{ for some }i\in [I]\right\}.
\end{align*}
Using the fact that $\I\left\{{\rxi}:\bm{a}(\bm{x})^{\top}{\rxi}_i\leq b_i(\bm{x}), \forall \in [I]\right\}+\I\left\{{\rxi}:\bm{a}(\bm{x})^{\top}{\rxi}_i>b_i(\bm{x}), \textrm{ for some }i\in [I]\right\}=1$, set $\hat{Z}$ is equivalent to
\begin{align*}
\hat{Z}
=&\left\{\bm{x}:\frac{1}{N}\sum_{j\in [N]}\inf_{\bm{\xi}:\|\bm{\xi}-\bzeta^j\|\leq \delta}\I\left\{\bm{a}(\bm{x})^{\top}{\rxi}_i\leq b_i(\bm{x}), \forall \in [I]\right\}\geq 1-\epsilon \right\}.
\end{align*}
Clearly,
\begin{align*}
\inf_{\bm{\xi}:\|\bm{\xi}-\bzeta^j\|\leq \frac{\delta}{\epsilon}}\I\left\{{\rxi}:\bm{a}(\bm{x})^{\top}{\rxi}_i\leq b_i(\bm{x}), \forall \in [I]\right\}
=\begin{cases}
1,&\textrm{ if } \frac{\delta}{\epsilon}\|\bm{a}(\bm{x})\|_*+\bm{a}(\bm{x})^{\top}{\bzeta}_i^j\leq b_i(\bm{x}),\forall \in [I]\\
0,&\textrm{ otherwise }.
\end{cases}, \forall j\in [N].
\end{align*}
Thus, 
\begin{align*}
\hat{Z}
=&\left\{\bm{x}:\frac{1}{N}\sum_{j\in [N]}\I\left\{\frac{\delta}{\epsilon}\|\bm{a}(\bm{x})\|_*+\bm{a}(\bm{x})^{\top}{\bzeta}_i^j\leq b_i(\bm{x}), \forall \in [I]\right\}\geq 1-\epsilon \right\}.\qed
\end{align*}

\end{proof}
This result demonstrates that set $Z_{\VaR}$ indeed can be viewed as a deterministic counterpart of DRCCP with $\infty-$Wasserstein ambiguity set. Thus, in practice, it can serve as an alternative for the set $Z$.
}

For the completeness of this paper, we present the mixed integer program formulation of outer approximation set $Z_{\VaR}$. The proof is omitted as it directly follows the proof of Theorem~\ref{cor_reform_multi_drccp}.
\begin{corollary}\label{cor_reform_multi_drccp_var}
	Suppose that there exists an $\bm M\in \Re^{N}$ such that 
	$$\max_{i\in [I]}\max_{\bm{x} \in Z_{\VaR}}\left\{\bm{a}(\bm{x})^{\top}\bm{\zeta}_i^j+\frac{\delta}{\epsilon}\|\bm{a}(\bm{x})\|_*-b_i(\bm{x})\right\}\leq M_{j}$$
	for all $j\in [N]$. Then $ \zsetd_{\VaR}$ is mixed integer representable, i.e.,
	\begin{subequations}\label{zset_P_W_I_1+_bigM_var}
		\begin{empheq}[left={\zsetd_{\VaR}=\empheqlbrace{\bm{x} \in \Re^n:}},right=\empheqrbrace]{align}
		& \frac{1}{N}\sum_{j\in [N]}y_j\geq 1-\epsilon,\\
		&\frac{\delta}{\epsilon}\|\bm{a}(\bm{x})\|_*\leq b_i(\bm{x})-\bm{a}(\bm{x})^{\top}\bm{\zeta}_i^j +M_{j}(1-y_{j}),\forall i\in [I], j\in [N],\\
		& y_{j}\in \{0,1\} ,\forall j\in [N].
		\end{empheq}
	\end{subequations}
\end{corollary}
\bl{Similar to Theorem~\ref{cor_reform_multi_drccp}, suppose that $\bm{x}\in [\bm{L},\bm{U}]$, then for each $j\in [N]$, one possible $M_j$ can be derived as below:
\begin{align*}
M_j:=\max_{i\in [I]}&\left\{\sum_{\tau\in \hat{S}_+}(B_{i\tau}-\eta_1 \zeta_{i\tau}^j)U_\tau+\sum_{\tau\in \hat{S}_-}(B_{i\tau}-\eta_1 \zeta_{i\tau}^j)L_\tau+b^i-\eta_2\zeta_{i(n+1)}^j\right.\\
&\left.+\frac{\delta}{\epsilon}\max\left\{\left\|\begin{pmatrix}
\eta_1 \bm{L}\\
\eta_2
\end{pmatrix}\right\|_*,\left\|\begin{pmatrix}
\eta_1 \bm{U}\\
\eta_2
\end{pmatrix}\right\|_*\right\}\right\}.
\end{align*}
where $\hat{S}_+=\{\tau\in [n]: B_{i\tau}-\eta_1 \zeta_{i\tau}^j>0\}$ and $\hat{S}_-=[n]\setminus\hat{S}_+$.}

\subsection{Inner Approximation I- Robust Scenario Approximation} 

We also observe that 
for any random variable $\tilde{X}$, we have
$$\CVaR_{1-\epsilon}\left(\tilde{X}\right)\leq \CVaR_{1}\left(\tilde{X}\right):=\mathrm{ess.}\sup(\tilde{X}).$$
Thus, in Corollary~\ref{ref_math_thm_cor}, if we replace $\CVaR_{1-\epsilon}\left(\cdot\right)$ by $\mathrm{ess.}\sup(\cdot)$, then we have the following inner approximation of set $\zsetd$. 
\begin{theorem}\label{cor_var_SA}Set $\zsetd$ can be inner approximated by
	\begin{empheq}[left={Z_R=\empheqlbrace{\bm{x}\in \Re^n:}},right=\empheqrbrace.]{align}
	& \frac{\delta}{\epsilon}\|\bm{a}(\bm{x})\|_*+\bm{a}(\bm{x})^{\top}\bm{\zeta}^j\leq b_i(\bm{x}), \forall j\in [N], i\in [I]\label{zset_P_W_inner}
	\end{empheq}
\end{theorem}
\begin{proof}
	Since $\CVaR_{1-\epsilon}\left[-f(\bm x,\bzeta)\right] \leq \mathrm{ess.}\sup\left[-f(\bm x,\bzeta)\right]$, and $\bzeta$ is a discrete random vector, therefore, set $\zsetd$ can be inner approximated by
	\begin{equation*}
	Z_R=\left\{\bm{x}\in \Re^n:\Pr_{\tilde{\bm{\zeta}}}\left\{f(\bm x,\bzeta)\geq \frac{\delta}{\epsilon}\right\}=1\right\}.
	\end{equation*}
Using the definition of $f(\bm x,\bm\zeta)$ and the fact that $\frac{\delta}{\epsilon}>0$, we arrive at \eqref{zset_P_W_inner}. \qed
\end{proof}
We remark that set $Z_R$ in \eqref{zset_P_W_inner} is very similar to scenario approach to regular chance constrained program \cite{calafiore2006scenario,campi2009scenario,nemirovski2006scenario}. That is, we generate $N$ i.i.d. samples $\{\bm{\zeta}^j\}_{j\in [N]}$ and enforce all the sampled constraints to hold. It has been shown in \cite{calafiore2006scenario,campi2009scenario,nemirovski2006scenario} that if $N$ is larger than a threshold, it guarantees with high probability that the solution of scenario approach is feasible to the regular chance constrained program.
Different from scenario approach, in formulation \eqref{zset_P_W_inner}, we add a penalty $\frac{\delta}{\epsilon}\|\bm{a}(\bm{x})\|_*$ to the sampled constraints, which can be viewed as a ``robust" scenario approach to the regular chance constrained problem. That is, if the sample size $N$ is not sufficiently large (i.e., $N$ is smaller than the threshold given by \cite{calafiore2006scenario,campi2009scenario,nemirovski2006scenario}), one might want to add a penalty $\frac{\delta}{\epsilon}\|\bm{a}(\bm{x})\|_*$ to enforce that set $Z_R$ is indeed a subset of the feasible region induced by a regular chance constraint.
\\

\subsection{Inner Approximation II- An Inner Chance Constrained  Programming Approximation} 
Next we propose an inner chance constrained programming approximation of set $Z$ by constructing a feasible $\gamma$ in \eqref{zset_P_W_2}.
\begin{theorem}\label{ref_math_thm_inner_2} Set $\zsetd$ is inner approximated by
	\begin{equation}
	\zsetd_{I}=\left\{\bm{x}\in \Re^n:\Pr_{\bzeta}\left\{\frac{\delta}{\epsilon-\alpha}\|\bm{a}(\bm{x})\|_*+\bm{a}(\bm{x})^{\top}\bzeta_i\leq b_i(\bm{x}), i\in [I]\right\}\geq 1-\alpha, 0\leq \alpha<\epsilon\right\}.\label{zset_inner_2_nc}
	\end{equation}
\end{theorem}
\begin{proof}
According to the definition of $f(\bm{x},\bm{\zeta})$ in \eqref{eq_def_f} and the fact $\frac{\delta}{\epsilon}>0$, set $\zsetd_{I}$ is equivalent to
\begin{equation}
	\zsetd_{I}=\left\{\bm{x}\in \Re^n:\Pr_{\bzeta}\left\{f(\bm x,\bzeta)\geq \frac{\delta}{\epsilon-\alpha}\right\}\geq 1-\alpha, 0\leq \alpha<\epsilon\right\}. \label{zset_inner_2}
	\end{equation}
For any $\bm{x}\in Z_I$, we need to show that $\bm{x}\in Z$. Since $\bm{x}\in Z_I$, there exists an $\alpha$ such that $(\bm{x},\alpha)$ satisfies constraints in \eqref{zset_inner_2}. Now let us define $\gamma=\frac{\delta}{\epsilon-\alpha}$. It remains to show that
$(\gamma,\bm{x})$ satisfies the constraints \eqref{zset_P_W_2}.

Let us define a set
\[\C=\left\{j\in [N]:f(\bm x,\bm\zeta^j)< \gamma\right\}.\]
Since $\gamma=\frac{\delta}{\epsilon-\alpha}$ and $\Pr_{\bzeta}\left\{\bzeta=\bm\zeta^j\right\}=\frac{1}{N}$ for all $j\in [N]$, thus according to \eqref{zset_inner_2}, we have 
$$|\C| \leq \sum_{j\in [N]}\I\left(f(\bm x,\bm\zeta^j)< \frac{\delta}{\epsilon-\alpha}\right)\leq N\alpha,$$
where the first inequality is due to $\gamma= \frac{\delta}{\epsilon-\alpha}$ and the second inequality is due to \eqref{zset_inner_2}. 
Hence,
\begin{align*}
 \frac{1}{N}\sum_{j\in [N]} \min\left\{f(\bm{x},\bm{\zeta}^j)-\gamma,0\right\} = \frac{1}{N}\sum_{j\in \C} \left(f(\bm{x},\bm{\zeta}^j)-\gamma\right)
 \geq  -\frac{|\C|}{N}\gamma\geq -\alpha\gamma =\delta-\epsilon\gamma,
\end{align*}
where the first inequality is due to $f(\bm{x},\bm{\zeta}^j)\geq 0$ according to its definition in \eqref{eq_def_f} and the second inequality is due to $|\C| \leq N\alpha$.\qed
\end{proof}

We remark that this result together with set $Z_{\VaR}$ shows that the DRCC set $Z$ can be inner and outer approximated by sets induced by regular chance constraints with empirical distribution $\Pr_{\bzeta}$.

We also observe that (i) set $Z_R$ is a special case of set $\zsetd_{I}$ by letting $\alpha=0$, thus, we must have $Z_R\subseteq \zsetd_{I}$; (ii) there are $\lceil N\epsilon\rceil$ disjoint intervals that $\alpha$ belong to, that is,
$$\alpha\in \cup_{i\in [\lceil N\epsilon\rceil]}\left[\frac{i-1}{N},\frac{i}{N}\right).$$ Suppose that $\alpha\in (\frac{i-1}{N}, \frac{i}{N})$ for some $i \in [\lceil N\epsilon\rceil]$. Since $\Pr_{\bzeta}\left\{\bzeta=\bm\zeta^j\right\}=\frac{1}{N}$ for all $j\in [N]$, thus the chance constraint in \eqref{zset_inner_2} is equivalent to
\[\Pr_{\bzeta}\left\{f(\bm x,\bzeta)\geq \frac{\delta}{\epsilon-\alpha}\right\}\geq 1-\frac{i-1}{N}.\]
The feasible region induced by the above chance constraint increases if we decrease the value of $\alpha$ to $\frac{i-1}{N}$. Therefore, to optimize over set $S\cap Z_{I}$, we only need to enumerate these $\lceil N\epsilon\rceil$ different values of $\alpha$ and choose the one which yields the smallest objective value; (iii) for each given $\alpha$, the chance constraint in \eqref{zset_inner_2_nc} is mixed integer program representable. These three results are summarized below.
\begin{corollary}\label{ref_math_thm_inner_2_cor1}Let set $\zsetd_{I}$ be defined in \eqref{zset_inner_2}, then
\begin{enumerate}[(i)]
\item $Z_R\subseteq Z_I\subseteq Z$;  
\item set $\zsetd_{I}=\cup_{\alpha\in \left\{0,\frac{1}{N},\ldots, \frac{\lceil N\epsilon\rceil-1}{N}\right\}}\zsetd_{I}^{\alpha}$, where set $\zsetd_{I}^{\alpha}$ is defined as
\begin{equation}
	\zsetd_{I}^\alpha=\left\{\bm{x}\in \Re^n:\Pr_{\bzeta}\left\{\frac{\delta}{\epsilon-\alpha}\|\bm{a}(\bm{x})\|_*+\bm{a}(\bm{x})^{\top}\bzeta_i\leq b_i(\bm{x}), i\in [I]\right\}\geq 1-\alpha\right\};\label{zset_inner_3}
	\end{equation}
for each $\alpha\in \left\{0,\frac{1}{N},\ldots, \frac{\lceil N\epsilon\rceil-1}{N}\right\}$; and
\item suppose that there exists an $\bm{M}^\alpha\in \Re^{N}$ such that 
	$$\max_{i\in [I]}\max_{\bm{x} \in Z_{I}^\alpha}\left\{\bm{a}(\bm{x})^{\top}\bm{\zeta}_i^j+\frac{\delta}{\epsilon-\alpha}\|\bm{a}(\bm{x})\|_*-b_i(\bm{x})\right\}\leq M_{j}^\alpha$$
	for all $j\in [N]$, then set $\zsetd_{I}^{\alpha}$ is mixed integer representable, i.e.,
	\begin{subequations}\label{zset_inner_alpha}
		\begin{empheq}[left={\zsetd_{I}^\alpha=\empheqlbrace{\bm{x} \in \Re^n:}},right=\empheqrbrace]{align}
		& \frac{1}{N}\sum_{j\in [N]}y_j\geq 1-\alpha,\\
		&\frac{\delta}{\epsilon-\alpha}\|\bm{a}(\bm{x})\|_*\leq b_i(\bm{x})-\bm{a}(\bm{x})^{\top}\bm{\zeta}_i^j +M_{j}^\alpha(1-y_{j}),\forall i\in [I], j\in [N],\\
		& y_{j}\in \{0,1\} ,\forall j\in [N].
		\end{empheq}
	\end{subequations}
\end{enumerate}
\end{corollary}
\bl{Similar to Theorem~\ref{cor_reform_multi_drccp}, suppose that $\bm{x}\in [\bm{L},\bm{U}]$, then for each $j\in [N]$ and $\alpha \in \left\{0,\frac{1}{N},\ldots, \frac{\lceil N\epsilon\rceil-1}{N}\right\}$, one possible $M_j^\alpha$ in Corollary~\ref{ref_math_thm_inner_2_cor1} can be derived as below:
\begin{align*}
M_j^\alpha:=\max_{i\in [I]}&\left\{\sum_{\tau\in \hat{S}_+}(B_{i\tau}-\eta_1 \zeta_{i\tau}^j)U_\tau+\sum_{\tau\in \hat{S}_-}(B_{i\tau}-\eta_1 \zeta_{i\tau}^j)L_\tau+b^i-\eta_2\zeta_{i(n+1)}^j\right.\\
&\left.+\frac{\delta}{\epsilon-\alpha}\max\left\{\left\|\begin{pmatrix}
\eta_1 \bm{L}\\
\eta_2
\end{pmatrix}\right\|_*,\left\|\begin{pmatrix}
\eta_1 \bm{U}\\
\eta_2
\end{pmatrix}\right\|_*\right\}\right\}.
\end{align*}
where $\hat{S}_+=\{\tau\in [n]: B_{i\tau}-\eta_1 \zeta_{i\tau}^j>0\}$ and $\hat{S}_-=[n]\setminus\hat{S}_+$.}

According to Corollary~\ref{ref_math_thm_inner_2_cor1}, to solve the inner approximation of DRCCP (i.e., $\min_{\bm{x}\in S\cap Z_I}$), we can solve  $\min_{\bm{x}\in S\cap Z_I^\alpha}$ for each $\alpha\in \left\{0,\frac{1}{N},\ldots, \frac{\lceil N\epsilon\rceil-1}{N}\right\}$ and choose the smallest value.

Finally, suppose that the empirical distribution $\Pr_{\bzeta}$ will converge to the true distribution $\Pr^{\infty}$ with an exponential rate (cf.,  Theorem 3.4 \cite{esfahani2015data}), i.e., if $N\rightarrow \infty$, then $\delta\rightarrow 0$ with rate $\delta=\frac{c_1}{N^{c_2}}$, where $c_1>0,c_2>0$ are positive constant. Then with probability one, $Z_{I}\rightarrow Z$ as $N\rightarrow\infty$. Indeed, suppose that $N$ is sufficiently large such that $\frac{c_1}{N^{\frac{c_2}{2}}}<1$. In \eqref{zset_inner_3}, let $\alpha=\frac{\lceil N\epsilon\rceil-\lceil c_1N^{1-\frac{c_2}{2}}\rceil-1}{N}$. Clearly, as $N\rightarrow \infty$, we have $\alpha\rightarrow \epsilon$  and
\[\frac{\delta}{\epsilon -\alpha}=\frac{c_1N^{1-c_2}}{N\epsilon +1-\lceil N\epsilon\rceil+\lceil c_1N^{1-\frac{c_2}{2}}\rceil} \leq N^{-\frac{c_2}{2}} \rightarrow 0\]
where the inequality is due to $N\epsilon +1\geq \lceil N\epsilon\rceil$ and $\lceil c_1N^{1-\frac{c_2}{2}} \rceil\geq c_1N^{1-\frac{c_2}{2}}$. This observation is summarized below.
\begin{proposition}Suppose that the empirical distribution $\Pr_{\bzeta}$ will converge to the true distribution $\Pr^{\infty}$ with an exponential rate. Then with probability one, we have
$Z_{I}\rightarrow Z$ as $N\rightarrow\infty$.
\end{proposition}
\bl{ We make the following two remarks:
\begin{itemize}
\item According to \cite{esfahani2015data}, any light-tail distribution (e.g., Gaussian distribution) satisfies the assumption in above proposition; and 
\item Sets $Z_{\VaR}$ and $Z_{I}$ together build up a hierarchy of regular chance constrained programs, which converges to DRCC set $Z$ as $N\rightarrow\infty$ and preserves the outer and inner approximations, i.e., $Z_{I}\subseteq Z\subseteq Z_{\VaR}$ for all $N$ and $Z_{I}\rightarrow Z$ and $Z_{\VaR}\rightarrow Z$ as $N\rightarrow\infty$.
\end{itemize}}

\subsection{Inner Approximation III- $\CVaR$ Approximation} 

In this subsection, we will study a well-known convex approximation of a chance constraint, which is to replace the nonconvex chance constraint by a convex constraint defined by $\CVaR$ (cf., \cite{nemirovski2006convex}). For DRCC set $Z$, the resulting approximation is
\begin{empheq}[left={\zsetd_{\CVaR}=\empheqlbrace{\bm{x}\in \Re^n:}},right=\empheqrbrace.]{align}
& \sup_{\Pr\in \P}\inf_{\beta}\left[-\epsilon \beta+\E_{\Pr}\left[\left(\max_{i\in [I]}\left(\bm{a}(\bm{x})^{\top}{\rxi_i}-b_i(\bm{x})\right)+\beta\right)_{+}\right]\right] \leq 0\label{sp_CVaR}
\end{empheq}
Set $\zsetd_{\CVaR}$ (\ref{sp_CVaR}) is convex and is an inner approximation of set $Z$. The following results show a reformulation of set $\zsetd_{\CVaR}$. We would like to acknowledge that this result has been independently observed by
a recent work in \cite{hota2018data}. \bl{Thus, the proof is omitted.} 
\begin{theorem}\label{thm_cvar} Set $\zsetd_{\CVaR}\subseteq Z$ is equivalent to
	\begin{subequations}\label{zset_P_W_inner_c}
		\begin{empheq}[left={\zsetd_{\CVaR}=\empheqlbrace{\bm{x}\in \Re^n:}},right=\empheqrbrace]{align}
		& \delta \nu-\epsilon\gamma \leq \frac{1}{N}\sum_{j\in [N]}z_j,\label{eq_gamma_nu_cvar}\\
		&z_j+\gamma\leq b_i(\bm{x})-\bm{a}(\bm{x})^{\top}\bm{\zeta}_i^j, \forall j\in [N],i\in [I],\label{eq_z_gamma_x_cvar}\\
		&z_j\leq 0, \forall j\in [N],\label{eq_z_0_cvar}\\
		&\|\bm{a}(\bm{x})\|_* \leq \nu, \label{eq_a_nu_cvar}\\
		&\nu\geq 0,\gamma\geq 0.
		\end{empheq}
	\end{subequations}
\end{theorem}

We remark that one can directly derive the equivalent form \eqref{zset_P_W_inner_c} of set $Z_{\CVaR}$ based upon formulation \eqref{zset_P_W_2}.
\begin{remark}
Since $\max\{b_i(\bm{x})-\bm{a}(\bm{x})^{\top}\bm{\zeta}_i,0\}\geq b_i(\bm{x})-\bm{a}(\bm{x})^{\top}\bm{\zeta}_i$, by replacing $\max\{b_i(\bm{x})-\bm{a}(\bm{x})^{\top}\bm{\zeta}_i,0\}$ with $b_i(\bm{x})-\bm{a}(\bm{x})^{\top}\bm{\zeta}_i$, then function $f(\bm{x},\bm{\zeta})$ is lower bounded by \[f(\bm{x},\bm{\zeta})\geq \underline{f}(\bm{x},\bm{\zeta})=\min\left\{\min_{i\in [I]\setminus\Ie(\bm x)}\frac{b_i(\bm{x})-\bm{a}(\bm{x})^{\top}\bm{\zeta}_i}{\|\bm{a}(\bm{x})\|_*},\min_{i\in \Ie(\bm x)}\chi_{b_i(\bm{x})<0}(x)\right\}.\]
Thus, set $Z$ can be inner approximated by the following set	
\begin{empheq}[left={\empheqlbrace{\bm{x}\in \Re^n:}},right=\empheqrbrace]{align*}
& \delta-\epsilon\gamma\leq \frac{1}{N}\sum_{j\in [N]} \min\left\{\underline{f}(\bm{x},\bm{\zeta}^j)-\gamma,0\right\} ,\\
&\gamma\geq 0.
\end{empheq}
By introducing additional variables $\bm{z}$ to linearize the nonlinear function $\min\left\{\underline{f}(\bm{x},\bm{\zeta})-\gamma,0\right\}$, we arrive at
\begin{empheq}[left={\empheqlbrace{\bm{x}\in \Re^n:}},right=\empheqrbrace]{align*}
& \delta-\epsilon\gamma\leq \frac{1}{N}\sum_{j\in [N]} z_j ,\\
		&z_j+\gamma\leq \min_{i\in [I]\setminus\Ie(\bm x)}\frac{b_i(\bm{x})-\bm{a}(\bm{x})^{\top}\bm{\zeta}_i}{\|\bm{a}(\bm{x})\|_*}, \forall j\in [I],\\
		&z_j+\gamma\leq \min_{i\in \Ie(\bm x)}\chi_{b_i(\bm{x})<0}(x), \forall j\in [I],\\
&\gamma\geq 0,z_j\leq 0, \forall j\in [n].
\end{empheq}
This set can be proven to be exactly equal to set $Z_{\CVaR}$ by discussing whether $\|\bm{a}(\bm{x})\|_*>0$ or not: 
\begin{enumerate}[(i)]
\item if $\|\bm{a}(\bm{x})\|_*>0$, then replace $\delta, \{z_j\}_{j\in [N]}$ and $\gamma$ by $\delta\|\bm{a}(\bm{x})\|_*, \{z_j\|\bm{a}(\bm{x})\|_*\}_{j\in [N]}$ and $\gamma\|\bm{a}(\bm{x})\|_*$;
\item if $\|\bm{a}(\bm{x})\|_*=0$, since $\delta\leq \epsilon\gamma+\frac{1}{N}\sum_{j\in [N]} z_j$ and $\delta>0,\epsilon>0$, according to the pigeonhole principle, we must have $z_{j_0}+\gamma>0$ for some $j_0\in [N]$, which implies that $b_i(\bm{x})\geq 0$ for all $i\in [I]$.
\end{enumerate}
\end{remark}

This observation inspires us that $Z_{\CVaR}=Z$ if $N\epsilon \leq 1$. In fact, if $N\epsilon \leq 1$, then we must have $f(\bm{x},\bm{\zeta}^j)= \underline{f}(\bm{x},\bm{\zeta}^j)$ for all $j\in [N]$, which implies that $Z_{\CVaR}=Z$.
\begin{proposition}\label{thm_cvar_normed2}Suppose that $\epsilon \in (0,1/N]$, then $\zsetd=\zsetd_{\CVaR}$.
\end{proposition}
\begin{proof}
We note that $\zsetd_{\CVaR}\subseteq Z=Z_1\cup Z_2$, where $Z_1$ and $Z_2$ are defined in \eqref{zset_P_W_I_1-M} and \eqref{zset_P_W_I_2-M}, respectively. We note that set $Z_2\subseteq Z_{\CVaR}$. Indeed, suppose that $\bm{x}\in Z_2$, i.e., $\bm{a}(\bm{x})=\bm0,b_i(\bm{x})\geq 0$ for each $i\in [I]$, then let $\nu=0,\gamma=0$ and $z_j=0$ for each $j\in [N]$. Clearly, $(\nu,\gamma,\bm{z},\bm{x})$ satisfies the constraints in \eqref{zset_P_W_inner_c}. Hence, $\bm{x}\in Z_{\CVaR}$.

Thus, it is sufficient to show that $Z_1\subseteq \zsetd_{\CVaR}$. Indeed,  given $\bm{x}\in Z_1$, there exists $(\nu,\gamma,\bm{z})$ such that $(\nu,\gamma,\bm{z},\bm{x})$ satisfies the constraints in \eqref{zset_P_W_I_1-M}. We only need to show that $z_j+\gamma >0$ for each $j\in [N]$. Suppose that there exists a $j_0\in [N]$ such that $z_{j_0}+\gamma\leq 0$. Then according to \eqref{z_v_I_1-M}, we have
\[ \delta\nu\leq \frac{1}{N}\sum_{j\in [N]\setminus\{j_0\}}z_j +\frac{1}{N}(N\epsilon\gamma +z_{j_0}) \leq 0\]
where the second inequality is due to $\epsilon N\leq 1$ and $z_{j_0}+\gamma\leq 0$, a contradiction that $\delta>0$. Therefore, in \eqref{z_lambda_a_b_I_1-M}, we must have $$\max\left\{b_i(\bm{x})-\bm{a}(\bm{x})^{\top}\bm{\zeta}_i^j,0\right\}=b_i(\bm{x})-\bm{a}(\bm{x})^{\top}\bm{\zeta}_i^j$$ 
for each $i\in [I], j\in [N]$. Hence, $(\nu,\gamma,\bm{z},\bm{x})$ satisfies the constraints in \eqref{zset_P_W_inner_c}, i.e., $\bm{x}\in Z_{\CVaR}$.
\qed
\end{proof}
The result in Proposition \ref{thm_cvar_normed2} shows that if the risk parameter $\epsilon$ is small enough (i.e., less than or equal to $\frac{1}{N}$), then set $Z$ is convex and is equivalent to its $\CVaR$ approximation.

\subsection{Formulation Comparisons} 

First, we would like to compare sets $Z_R,Z_{\CVaR}$. Indeed, we can show that 
$Z_R\subseteq Z_{\CVaR}$, i.e., set $Z_R$ is at least as conservative as $\CVaR$ approximation $Z_{\CVaR}$.
\begin{proposition} Let $ Z_R,Z_{\CVaR}$ be defined in \eqref{zset_P_W_inner}, \eqref{zset_P_W_inner_c} , respectively. Then $Z_R\subseteq  Z_{\CVaR}.$
\end{proposition}
\begin{proof}
Given $\bm{x}\in Z_R$, we only need to show that $\bm{x}\in Z_{\CVaR}$. Indeed, let us consider $\nu=\|\bm{a}(\bm{x})\|_*$, $\gamma=\frac{\delta}{\epsilon}\|\bm{a}(\bm{x})\|_*,z_j=0$ for all $j\in [N]$, then we see that $(\nu,\gamma,\bm z,\bm x)$ satisfies the constraints in \eqref{zset_P_W_inner_c}, i.e., $\bm{x}\in Z_{\CVaR}$.\qed

\end{proof}

The following example illustrates sets $Z, Z_{\VaR}, Z_{\CVaR}, Z_R, Z_I$ and their inclusive relationships. 
\begin{example}\label{example1}Suppose $N=3,n=2,I=2, \delta=1/6,\epsilon=2/3 $ and $\bm{\zeta}_1^1=(0,0,\sqrt{2})^{\top}, \bm{\zeta}_2^1=(0,0,3\sqrt{2})^{\top},\bm{\zeta}_1^2=(0,0,3\sqrt{2})^{\top}, \bm{\zeta}_2^2=(0,0,\sqrt{2})^{\top},\bm{\zeta}_1^3=(0,0,3\sqrt{2})^{\top}, \bm{\zeta}_2^3=(0,0,2\sqrt{2})^{\top},\bm{a}(\bm{x})=\e_3=\begin{pmatrix}
0,0,1
\end{pmatrix}^\top, b_1(x)=x_1,b_2(x)=x_2$. Then, \eqref{zset} becomes:
\begin{equation}
\zsetd:=\left\{(x_1,x_2):\inf_{\Pr\in \P}\Pr\left\{(\tilde{\rxi}_1,\tilde{\rxi}_2):\tilde{\xi}_{13} \leq x_1, \tilde{\xi}_{23} \leq x_2\right\} \ge \frac{1}{3}\right\}. \label{zset_exp}
\end{equation}

By straightforward calculation, we have
	\begin{align*}
	Z&=\left\{(x_1,x_2):2+\frac{\sqrt{2}}{2}\leq x_1, 3+\frac{\sqrt{2}}{2}\leq x_2\right\}\cup\left\{(x_1,x_2):3+\frac{\sqrt{2}}{2}\leq x_1, 2+\frac{\sqrt{2}}{2}\leq x_2\right\}\\
&\cup \left\{(x_1,x_2):3\leq x_1, 3\leq x_2,6+\frac{\sqrt{2}}{2}\leq x_1+ x_2\right\}\\
Z_{\VaR}&=\left\{(x_1,x_2):2+\frac{\sqrt{2}}{4}\leq x_1, 3+\frac{\sqrt{2}}{4}\leq x_2\right\}\cup\left\{(x_1,x_2):3+\frac{\sqrt{2}}{4}\leq x_1, 2+\frac{\sqrt{2}}{4}\leq x_2\right\}\\
Z_{\CVaR}&=\left\{(x_1,x_2):3\leq x_1, 3\leq x_2,6+\frac{\sqrt{2}}{2}\leq x_1+ x_2\right\}\\
Z_R&=\left\{(x_1,x_2):3+\frac{\sqrt{2}}{4}\leq x_1, 3+\frac{\sqrt{2}}{4}\leq x_2\right\}\\
Z_{I}&=\left\{(x_1,x_2):2+\frac{\sqrt{2}}{2}\leq x_1, 3+\frac{\sqrt{2}}{2}\leq x_2\right\}\cup\left\{(x_1,x_2):3+\frac{\sqrt{2}}{2}\leq x_1, 2+\frac{\sqrt{2}}{2}\leq x_2\right\}\\
&\cup\left\{(x_1,x_2):3+\frac{\sqrt{2}}{4}\leq x_1, 3+\frac{\sqrt{2}}{4}\leq x_2\right\}.
	\end{align*}
	Clearly, we have $Z_R\subsetneq \left\{\begin{subarray}{c}
		Z_{\CVaR}\\
		\rotatebox[origin=c]{-90}{$\not\subseteq$}\\
			Z_{I}
		\end{subarray}\right\}\subsetneq Z\subsetneq Z_{\VaR} $ (see Figure~\ref{fig_exp2} for an illustration).
\end{example}

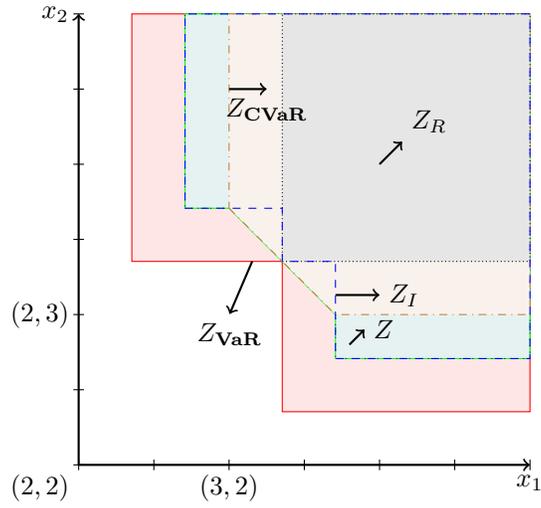
\begin{figure}[htbp]
\centering
	\begin{tikzpicture}[scale=2]
\draw[thick,->] (2,2) -- (5,2) node [below] {$x_1$};
		\draw[thick,->] (2,2) -- (2,5) node [left] {$x_2$};
		
		\foreach \x in {2,2.5,3,...,5} {%
			\draw ($(\x,2) + (0,-\TickSize)$) -- ($(\x,2) + (0,\TickSize)$);
		}
		
		\foreach \y in {2,2.5,3,...,5} {%
			\draw ($(2,\y) + (-\TickSize,0)$) -- ($(2,\y) + (\TickSize,0)$);
		}
		
		\node[below left] at (2,2) {$(2,2)$};
		\node[below] at (3,2) {$(3,2)$};	
		\node[left] at (2,3) {$(2,3)$};

		\draw[fill=red!10, draw=red] ({2+ sqrt(2)/4},{3+ sqrt(2)/4}) -- ({3+ sqrt(2)/4},{3+ sqrt(2)/4})--({3+ sqrt(2)/4},{2+ sqrt(2)/4})
 -- ({5},{2+ sqrt(2)/4})--(5,5)--({2+ sqrt(2)/4},5)-- cycle;
		\node[below] at ({2.65+ sqrt(2)/4},{2.65+ sqrt(2)/4}) {$Z_{\VaR}$};
		\draw[thick,->]({2.8+ sqrt(2)/4},{3+ sqrt(2)/4}) -- ({2.65+ sqrt(2)/4},{2.65+ sqrt(2)/4});

		\draw[fill=pinegreen!10, draw=green] ({2+ sqrt(2)/2},{3+ sqrt(2)/2}) -- ({3},{3+ sqrt(2)/2})-- ({3+ sqrt(2)/2},{3})--({3+ sqrt(2)/2},{2+ sqrt(2)/2})-- ({5},{2+ sqrt(2)/2})--(5,5)--({2+ sqrt(2)/2},5)-- cycle;

		\draw[draw=blue, dotted] ({2+ sqrt(2)/2},{3+ sqrt(2)/2}) -- ({3+ sqrt(2)/4},{3+ sqrt(2)/2})--({3+ sqrt(2)/4},{3+ sqrt(2)/4})-- ({3+ sqrt(2)/2},{3+ sqrt(2)/4})--({3+ sqrt(2)/2},{2+ sqrt(2)/2})-- ({5},{2+ sqrt(2)/2})--(5,5)--({2+ sqrt(2)/2},5)-- cycle;

       \draw[fill=brown!10, draw=brown, dashdotted] ({3},{3+ sqrt(2)/2}) -- ({3+ sqrt(2)/2},{3})--({5},{3})--(5,5)--(3,5)-- cycle;

\draw[fill=black!10,draw=black, densely dotted]  ({3+ sqrt(2)/4},{3+ sqrt(2)/4})--({5},{3+ sqrt(2)/4})--(5,5)--({3+ sqrt(2)/4},5)-- cycle;

		\node[right] at ({3.20+ 0.7},{2.20+ 0.7}) {$Z$};
		\draw[thick,->]({3+ 0.8},{2+ 0.8}) -- ({3.20+ 0.7},{2.20+ 0.7});

\draw[draw=blue, dashed] ({2+ sqrt(2)/2},{3+ sqrt(2)/2}) -- ({3+ sqrt(2)/4},{3+ sqrt(2)/2})--({3+ sqrt(2)/4},{3+ sqrt(2)/4})-- ({3+ sqrt(2)/2},{3+ sqrt(2)/4})--({3+ sqrt(2)/2},{2+ sqrt(2)/2})-- ({5},{2+ sqrt(2)/2})--(5,5)--({2+ sqrt(2)/2},5)-- cycle;
		\node[right] at ({4.0},{3.13}) {$Z_I$};
		\draw[thick,->]({3+ sqrt(2)/2},{3.13}) -- ({4.0},{3.13});

	\node[below] at ({3.25},{4.5}) {$Z_{\CVaR}$};
		\draw[thick,->]({3},{4.5}) -- ({3.25},{4.5});

\node[above right] at ({4.15},{4.15}) {$Z_R$};
		\draw[thick,->]({4},{4}) -- ({4.15},{4.15});
%
%
%
%
			

	\end{tikzpicture}
	\caption{Illustration of Example~\ref{example1}}
	\label{fig_exp2}
\end{figure}

Finally, the theoretical inclusive relationships of sets $Z, Z_{\VaR},  Z_R, Z_I,Z_{\CVaR}$ are shown in Figure~\ref{fig1} and their reformulations are summarized in Table~\ref{table_summary}.

\begin{figure}[htbp]
	\begin{center}
		\includegraphics[width=0.6\textwidth]{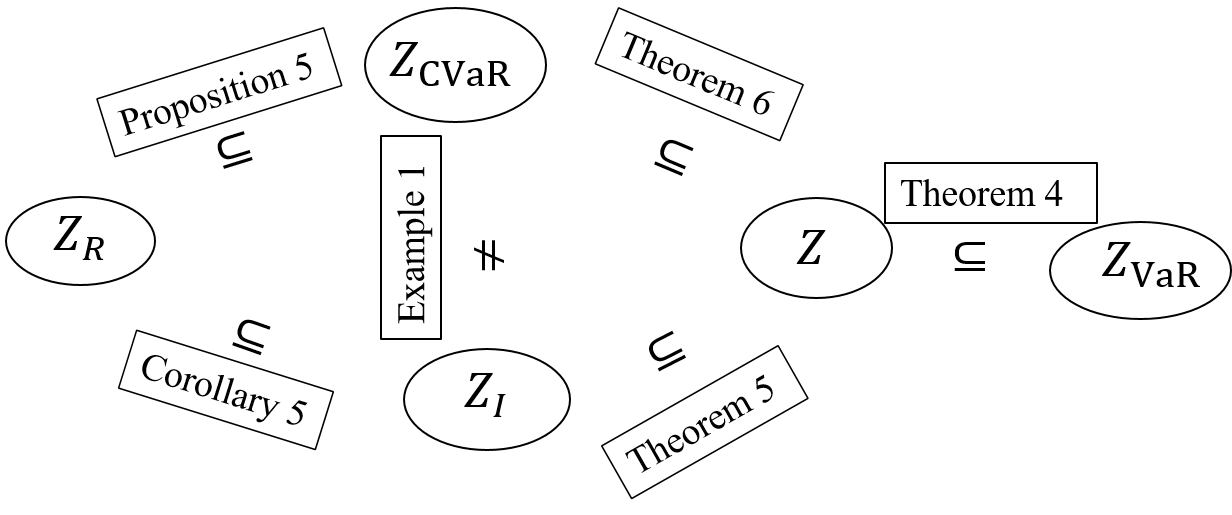} 
	\end{center}
	\caption{Summary of formulation comparisons}
	\label{fig1}
\end{figure}

\begin{table}[htbp]
	\centering
	\caption{Summary of exact formulation and inner and outer approximations from Sections~\ref{sec_general} and~\ref{sec_Approximation}}
	\label{table_summary}
	\renewcommand{\arraystretch}{1.15}
	\begin{tabular}{ccccc}
		\hline
	 Set $Z$ & Set $Z_{\VaR}$  & Set $Z_R$& Set $Z_{I}$ & Set $Z_{\CVaR}$  \\ \hline
	 Mixed-integer & Mixed-integer& Convex & Mixed-integer & Convex
	\\\hline
			 Theorem~\ref{cor_reform_multi_drccp} & Corollary~\ref{cor_reform_multi_drccp_var} & Theorem~\ref{cor_var_SA} & Corollary~\ref{ref_math_thm_inner_2_cor1}& Theorem~\ref{thm_cvar}  
	\\ \hline
	\end{tabular}
\end{table}

}
\section{DRCCP with Pure Binary Decision Variables}\label{sec_DRMKP}

In this section, we will study DRCCP with pure binary decision variables $\bm{x}\in \{0,1\}^n$, i.e., we assume that $S\subseteq  \{0,1\}^n$.
{If $S$ is a bounded integer set, we can use binary expansion to reformulate $S$ an equivalent binary set (c.f., \cite{zou2017stochastic}). }For binary DRCCP, we will show that the reformulations in the previous section can be improved.

\subsection{Polyhedral Results of Submodular Functions: A Review}
Our main derivation of stronger formulations is based upon some polyhedral results of submodular functions, which will be briefly reviewed in this subsection.

{We first briefly introduce the definition of submodularity and interested readers are referred to \cite{edmonds1970submodular,lovasz1983submodular} for more details.
\begin{definition}\label{def_subm_discrete}
{\rm\bf (Submodularity)} Let $2^{[n]}$ be \bl{the} power set of $[n]$. Then a set function $g: 2^{[n]}\rightarrow \mathbb{R}$ is ``\textit{submodular}" if and only if it satisfies the following condition:
\begin{itemize}
\item for every $T_1,T_2\subseteq [n] $ with $T_1 \subseteq T_2$ and every $t\in [n]\setminus T_2$, we must have $g(T_1\cup \{t\})-g(T_1)\geq g(T_2\cup \{t\})-g(T_2)$.
\end{itemize}
\end{definition}}

We first begin with the following lemmas on submodular functions.
\begin{lemma}\label{lemmasub} 
Given $\bm{d}_1\in \Re_+^n,d_2,d_3\in \Re$, function $f(\bm{x})=-\max\left(\bm{d}_1^{\top}\bm{x}+d_2, d_3\right)$ is submodular over the binary hypercube.
\end{lemma}
{\begin{proof} Since $\bm{d}_1^{\top}\bm{x}+d_2$ is a nondecreasing submodular function and $-\max\left(t, d_3\right)$ is \bl{a} nonincreasing concave function, the submodularity of their composition follows by Table 1 in \cite{topkis1978minimizing}.\qed
\end{proof}}
%

\begin{lemma}\label{lemmasub2}
Given $q\geq 1$, function $f(\bm{x})=\|\bm{x}\|_q$ with $q\geq 1$ is submodular over the binary hypercube.
\end{lemma}
\begin{proof}
This is because $f(\bm{x})=\|\bm{x}\|_q=\sqrt[q]{\sum_{l\in [n]}x_l}$, and $g(\e^{\top}\bm{x})$ is a submodular function if $g(\cdot)$ is a concave function (cf., \cite{yu2017polyhedral}).
\qed
\end{proof}

Next, we will introduce polyhedral properties of submodular functions. For any given submodular function $f(\bm{x})$ with $\bm{x}\in \{0,1\}^{n}$, let us denote $\Pi_{f}$ to be its epigraph, i.e.,
\[\Pi_{f}=\left\{(\bm{x},\phi): \phi\geq f(\bm{x}), \bm{x}\in \{0,1\}^{n}\right\}.\]
Then the convex hull of $\Pi_{f}$ is characterized by the system of \textit{``extended polymatroid inequalities"} (EPI) \cite{atamturk2008polymatroids,yu2017polyhedral}, i.e.,
\begin{align}
\conv\left(\Pi_{f}\right)=\left\{(\bm{x},\phi):f(\bm{0})+\sum_{l\in [n]}\rho_{\sigma_l}x_{\sigma_l}\leq \phi, \forall \sigma \in \Omega,\bm{x}\in [0,1]^{n}\right\},\label{eq_EPI}
\end{align}
where $\Omega$ denotes a collection of all permutations of set $[n]$ and $\rho_{\sigma_l}=f(\e_{A_l^{\sigma}})-f(\e_{A_{l-1}^{\sigma}})$ for each $l\in [n]$ with $A_0^{\sigma}=\emptyset, A_l^{\sigma}=\{\sigma_1,\ldots, \sigma_l\}$ and $(\e_{T})_{\tau}=\begin{cases}
1,& \text{if }\tau\in T\\
0,& \text{if }\tau\in [n]\setminus T
\end{cases}$.

In addition, although there are $n!$ number of inequalities in \eqref{eq_EPI}, these inequalities can be easily separated by a greedy procedure. 
\begin{lemma}\label{lemmasub3}(\cite{atamturk2008polymatroids,yu2017polyhedral})
Suppose $(\tilde{\bm{x}}, \tilde{\phi})\notin\conv\left(\Pi_{f}\right)$, and $\sigma\in \Omega$ be a permutation of $[n]$ such that $\tilde{x}_{\sigma_1} \geq \ldots\geq \tilde{x}_{\sigma_n} $. Then $(\tilde{\bm{x}}, \tilde{\phi})$ must violate the constraint $f(\bm{0})+\sum_{l\in [n]}\rho_{\sigma_l}x_{\sigma_l}\leq \phi$.
\end{lemma}
From Lemma~\ref{lemmasub3}, we see that to separate a point $(\tilde{\bm{x}}, \tilde{\phi})$ from $\conv\left(\Pi_{f}\right)$, we only need to sort the coordinates of $\tilde{\bm{x}}$ in a descending order, i.e., $\tilde{x}_{\sigma_1} \geq \ldots\geq \tilde{x}_{\sigma_n} $. Then $(\tilde{\bm{x}}, \tilde{\phi})$ can be the separated by the constraint $f(\bm{0})+\sum_{l\in [n]}\rho_{\sigma_l}x_{\sigma_l}\leq \phi$ from $\conv\left(\Pi_{f}\right)$. The time complexity of this separating procedure is $O(n\log n)$.

\subsection{Reformulating a Binary DRCCP by Submodular Knapsack Constraints: Big-M free}
In this section, we will replace the nonlinear constraints defining the feasible region of a binary DRCCP (i.e., set $S\cap Z$) by submodular knapsack constraints. These constraints can be equivalently described by the system of EPI in \eqref{eq_EPI}. Therefore we obtain a big-M free mixed integer representation of set $S\cap Z$.

First, we introduce $n$ auxiliary variables complementing binary variables $\bm{x}$, denoted by $\bm{w}$, i.e., $w_l+x_l=1$ for each $l\in [n]$. With these $n$ additional variables, we can reformulate function $b_i(\bm{x})-\bm{a}(\bm{x})^{\top}\bm{\zeta}_i^j$ as 
\begin{align}\label{eq_reform_b_a_zeta}
b_i(\bm{x})-\bm{a}(\bm{x})^{\top}\bm{\zeta}_i^j=\bm{r}_{ij}^{\top} \bm{x}+\bm{t}_{ij}^{\top} \bm{w}+u_{ij}
\end{align}                                                                                              
for each $i\in [I], j\in [N]$ such that $\bm{r}_{ij}\in \Re_+^n, \bm{t}_{ij}\in \Re_+^n$. {Indeed, since $\bm{a}(\bm{x})= \begin{pmatrix}
\eta_1\bm{x}\\
\eta_2
\end{pmatrix}$} and $b_i(\bm{x})=\bm B_i^{\top} \bm x+b^i$, in \eqref{eq_reform_b_a_zeta}, we can choose
\begin{align*}
r_{ijl}&=B_{il}\I(B_{il}>0)-\eta_1 \zeta_{il}^j \I(\zeta_{il}^j<0), \\
t_{ijl}&=-B_{il}\I(B_{il}<0)+\eta_1 \zeta_{il}^j \I(\zeta_{il}^j>0), \\
u_{ij}&=b^i-\eta_2\bm{\zeta}_{i(n+1)}^j+\sum_{\tau\in [n]}\left(B_{i\tau}\I(B_{i\tau}<0)-\eta_1\zeta_{i\tau}^j \I(\zeta_{i\tau}^j>0)\right),
\end{align*}
for each $l\in [n],i\in [I],j\in [N]$.

Thus, from above discussion, we can formulate $S\cap Z$ (recall that set $Z=Z_1\cup Z_2$ according to Proposition~\ref{thm_reform_joint_drccp}) as the following mixed integer set with submodular knapsack constraints.
\begin{theorem}\label{thm_reform_joint_drccp_binary} 
Suppose that $S\subseteq \{0,1\}^n$. Then $S\cap \zsetd=(S\cap\hat \zsetd_{1})\cup (S\cap\zsetd_{2})$, where
	\begin{subequations}\label{zset_P_W_I_1-M_binary}
		\begin{empheq}[left={S\cap \hat\zsetd_{1}=\empheqlbrace{\bm{x}\in S:}},right=\empheqrbrace]{align}
		& \delta \nu-\epsilon\gamma \leq \frac{1}{N}\sum_{j\in [N]}z_j,\label{z_v_I_1-M_binary}\\
		&-\max\left\{\bm{r}_{ij}^{\top} \bm{x}+\bm{t}_{ij}^{\top} \bm{w}+u_{ij},0\right\} \leq -z_j-\gamma, \forall i\in [I],j\in [N],\label{z_lambda_a_b_I_1-M_binary}\\
		&z_j\leq 0, \forall j\in [N],\label{z_0_I_1-M_binary}\\
		&\left\| \begin{pmatrix}
\eta_1 \bm{x}\\\eta_2
\end{pmatrix}\right\|_* \leq \nu,\label{a_v_I_1-M_binary}\\
& w_l+x_l=1, \forall l \in [n],\\
&\nu\geq 1, \label{v_geq_1}\\
		&\gamma\geq 0,\bm{w}\in \{0,1\}^n
		\end{empheq}
	\end{subequations}
	and
	\begin{empheq}[left={S\cap \zsetd_{2}=\empheqlbrace{\bm{x}\in S:}},right=\empheqrbrace]{align}
	\bm{a}(\bm{x})=\bm0,b_i(\bm{x})\geq 0,\forall i\in [I]\label{zset_P_W_I_2-M_binary}
	\end{empheq}
\end{theorem}
\begin{proof}According to Proposition~\ref{thm_reform_joint_drccp}, equalities \eqref{eq_reform_b_a_zeta} and the fact that $\bm{a}(\bm{x})=  \begin{pmatrix}
\eta_1 \bm{x}\\\eta_2
\end{pmatrix}$ with constant $\eta_1,\eta_2\in \{0,1\}$, constraints \eqref{z_lambda_a_b_I_1-M} and \eqref{a_v_I_1-M} {are} equivalent to \eqref{z_lambda_a_b_I_1-M_binary} and \eqref{a_v_I_1-M_binary}. Thus, we only need to show that $S\cap Z_1\subseteq(S\cap\hat \zsetd_{1})\cup (S\cap\zsetd_{2})$. 
There are two cases.
\begin{enumerate}[{Case} 1.]
\item If $\eta_2=1$, then we must have  $\|\bm{a}(\bm{x})\|_*=\left\| \begin{pmatrix}
\eta_1 \bm{x}\\\eta_2
\end{pmatrix}\right\|_*\geq 1$, then $S\cap Z_1=S\cap\hat{Z}_1$. We are done. 

\item If $\eta_2=0$, then we must have $\eta_1=1$. For any $\bm{x}\in S\cap Z_1 $, we need to show that $\bm{x}\in (S\cap\hat \zsetd_{1})\cup (S\cap\zsetd_{2})$. If $\bm{x}=0$, then the constraints \eqref{zset_P_W_I_1-M} become
\begin{align*}
		& \delta \nu\leq \frac{1}{N}\sum_{j\in [N]}z_j+\epsilon\gamma ,\\
		&z_j+\gamma\leq \max\left\{b_i(\bm{x}),0\right\}, \forall i\in [I],j\in [N],\\
		&z_j\leq 0, \forall j\in [N],\\
		&\nu>0,\gamma\geq 0.
\end{align*}
Since $\nu>0,\delta>0,1>\epsilon>0$, thus by the pigeonhole principle, we must have $z_{j_0}+\gamma>0$ for some $j_0\in [N]$. This implies that
$b_i(\bm{x})>0$ for each $i\in [I]$. Together with $\bm{a}(\bm{x})=\bm0$, we must have $\bm{x}=0\in S\cap Z_2$.

Now suppose that $\bm{x}\neq 0$. Note that $S\cap Z_1\subseteq \{0,1\}^n$, therefore, $\bm{x}\neq 0$ implies that $\|\bm{x}\|_*\geq 1$, thus, $v\geq \left\| \begin{pmatrix}
\eta_1 \bm{x}\\\eta_2
\end{pmatrix}\right\|_*=\|\bm{x}\|_* \geq 1$. Thus, $\bm{x}\in S\cap \hat{Z}_1$.\qed
\end{enumerate}
\end{proof}

From the proof of Theorem~\ref{thm_reform_joint_drccp_binary}, we note that if $b^i\geq \frac{\delta}{\epsilon}$ for each $i\in [I]$, then we have $S\cap Z_2\subseteq S\cap \hat Z_1$. Thus, $S\cap Z=S\cap \hat Z_1$.
\begin{corollary}\label{cor_reform_joint_drccp_binary1} 
Suppose that $S\subseteq \{0,1\}^n$ and $b^i\geq \frac{\delta}{\epsilon}$ for each $i\in [I]$. Then $ S\cap Z=S\cap \hat Z_1$.
\end{corollary} 
\begin{proof}
{From the proof of Theorem~\ref{thm_reform_joint_drccp_binary}, we only need to show that $\bm{x}=\bm0\in S\cap \hat Z_1$. In this case, we have $\bm{w}=\e-\bm{x}=\e$. Then according to \eqref{eq_reform_b_a_zeta}, we have $\bm{r}_{ij}^{\top} \bm{x}+\bm{t}_{ij}^{\top} \bm{w}+u_{ij}=b_i(\bm{x})-\bm{a}(\bm{x})^{\top}\bm{\zeta}_i^j=b^i$. Let us set $\nu=1,\gamma=\frac{\delta}{\epsilon}, \bm{z}=\bm0$. Then it is easy to see that $(\bm{x},\bm{w},\bm{z},\gamma,\nu)$ satisfies the constraints in \eqref{zset_P_W_I_1-M_binary}, i.e., $\bm 0\in S\cap \hat Z_1$.}
\qed
\end{proof}

We note that the left-hand sides of constraints \eqref{z_lambda_a_b_I_1-M_binary} and \eqref{a_v_I_1-M_binary} are submodular functions according to Lemma~\ref{lemmasub} and Lemma~\ref{lemmasub2}. Therefore, equivalently, we can replace these constraints with the convex hulls of epigraphs of their associated submodular functions. Thus, \bl{we arrive at the following equivalent representation of set $S\cap \hat{\zsetd}_{1}$.}
\begin{corollary}\label{cor_reform_joint_drccp_binary} 
Suppose that $S\subseteq \{0,1\}^n$ and $\|\cdot\|$ is $L_p$ norm with $p\geq 1$. Then 
	\begin{subequations}\label{cor_zset_P_W_I_1-M_binary}
		\begin{empheq}[left={S\cap \hat{\zsetd}_{1}=\empheqlbrace{\bm{x}\in S:}},right=\empheqrbrace]{align}
		& \delta \nu-\epsilon\gamma \leq \frac{1}{N}\sum_{j\in [N]}z_j,\label{cor_z_v_I_1-M_binary}\\
		&(\bm{x},\bm{w}, -z_j-\gamma) \in \conv(\Pi_{ij}), \forall i\in [I],j\in [N],\label{cor_z_lambda_a_b_I_1-M_binary}\\
		&z_j\leq 0, \forall j\in [N],\label{cor_z_0_I_1-M_binary}\\
		&(\bm{x},\nu) \in \conv(\Pi_{0}) ,\label{cor_a_v_I_1-M_binary}\\
& w_l+x_l=1, \forall l \in [n],\\
		&\nu\geq 1,\gamma\geq 0,\bm{w}\in [0,1]^n,
		\end{empheq}
	\end{subequations}
where
\begin{subequations}
\begin{align}
\Pi_{ij}&=\left\{(\bm{x},\bm{w}, \phi): -\max\left\{\bm{r}_{ij}^{\top} \bm{x}+\bm{t}_{ij}^{\top} \bm{w}+u_{ij},0\right\} \leq \phi, \bm{x}, \bm{w}\in \{0,1\}^n\right\},\forall i\in [I],j\in [N], \\
\Pi_{0}&{=\left\{(\bm{x},\phi): \left\| \begin{pmatrix}
\eta_1 \bm{x}\\\eta_2
\end{pmatrix}\right\|_* \leq \phi, \bm{x}\in \{0,1\}^n\right\}}
\end{align}
and $\left\{\conv(\Pi_{ij})\right\}_{i\in [I],j\in [N]}, \conv(\Pi_{0})$ can be described by the system of EPI in \eqref{eq_EPI}.
	\end{subequations}
\end{corollary} 

Note that the optimization problem $\min_{\bm{x}\in S\cap \zsetd_{1}}\bm{c}^{\top}\bm{x}$
can be solved by a branch and cut algorithm. In particular, at each branch and bound node, denoted as $(\hat{\bm{x}},\hat{\bm{w}},\hat{\bm{z}},\hat\gamma,\hat\nu)$, there might be too many (i.e., $N\times I+1$) valid inequalities to add, since in \eqref{cor_z_lambda_a_b_I_1-M_binary} and \eqref{cor_a_v_I_1-M_binary}, there are $N\times I+1$ convex hulls of epigraphs (i.e., $\left\{\conv(\Pi_{ij})\right\}_{i\in [I],j\in [N]}, \conv(\Pi_{0})$) to be separated from. Therefore, instead, we can first check and find the epigraphs of $\kappa$ ({e.g., $\kappa=10$} in our numerical study) most violated constraints in \eqref{z_lambda_a_b_I_1-M_binary} and \eqref{a_v_I_1-M_binary}, i.e., find the epigraphs corresponding to the $\kappa$ largest values in the following set
$$\left\{-\max\left\{\bm{r}_{ij}^{\top}\hat{ \bm{x}}+\bm{t}_{ij}^{\top} \hat{\bm{w}}+u_{ij},0\right\} +\hat z_j+\hat\gamma\right\}_{i\in [I],j\in [N]} \bigcup\left\{\left\| \begin{pmatrix}
\eta_1 \hat{\bm{x}}\\\eta_2
\end{pmatrix}\right\|_*- \hat{\nu}\right\}.$$ Finally, we can generate and add valid inequalities by separating $(\hat{\bm{x}},\hat{\bm{w}},\hat{\bm{z}},\hat\gamma,\hat\nu)$ from the convex hulls of these $\kappa$ epigraphs according to Lemma~\ref{lemmasub3}.

{
\section{Numerical Demonstration}\label{sec_sep_numerical}

In this section, we present a series of numerical studies to demonstrate the effectiveness of the proposed formulations and also show how to use cross validation to choose a proper Wasserstein radius $\delta$. 

For the demonstration purpose, we will study 
distributionally robust multidimensional knapsack problem (DRMKP) \cite{cheng2014distributionally,song14chance,Xie2016drccp} with continuous decision variables (i.e., continuous DRMKP) or binary decision variables (i.e., binary DRMKP). In a DRMKP, there are $n$ items and $I$ knapsacks. Additionally, $c_j$ represents the value of item $j$ for all $j \in [n]$, $\tilde\bxi_i := (\tilde\xi_{i1}, \ldots, \tilde\xi_{in})^{\top}$ represents the vector of random item weights in knapsack $i$, and $b^i>0$ represents the capacity limit of knapsack $i$, for all $i \in [I]$. The decision variable $x_j \in [0,1]$ represents the proportion of $j$th item to be picked. In a continuous DRMKP, we let $\bm{x}\in S:=[0,1]^n$, and for a binary DRMKP, we let $\bm{x}\in S:=\{0,1\}^n$. We use the Wasserstein ambiguity set under Assumptions \ref{assume_A1} and \ref{assume_A2} with $L_2$- norm as distance metric. With the notation above, DRMKP is formulated as
\begin{align}
v^* = \max_{\bm x\in S} \quad &\bm c^{\top}\bm x,\notag\\
\rm{s.t.}\quad&\inf_{\Pr\in \P}\Pr\left\{\tilde\rxi_i^{\top}\bm x\leq b^i,{\forall i\in [I]}\right\}\ge 1-\epsilon,\label{eq_drccp_con_mdk}
\end{align}
where the chance constraint here is to guarantee that the worst-case probability that each knapsack's capacity should be satisfied is at least $1-\epsilon$.

In the following subsections, we generated different random instances to test the proposed formulations. All the instances were executed on a MacBook Pro with a 2.80 GHz processor and 16GB RAM with a call of the commercial solver Gurobi (version 7.5, with default settings). \bl{We set the time limit of solving each instance to be 3600 seconds.}

\subsection{Continuous DRMKP: Numerical Demonstration of Exact Formulation, Outer and Inner Approximations}

In this subsection, we use continuous DRMKP (i.e., $S=[0,1]^n$ in \eqref{eq_drccp_con_mdk}) to numerically demonstrate the exact formulation in Theorem~\ref{cor_reform_multi_drccp}, outer approximation in Corollary~\ref{cor_reform_multi_drccp_var}, $\CVaR$ approximation in Theorem~\ref{thm_cvar} and inner chance constrained programming approximation in Corollary~\ref{ref_math_thm_inner_2_cor1}. To test the proposed formulations, we generated 10 random instances with $n=20$ and $I=10$, indexed by $\{1,2,\ldots,10\}$. For each instance, we generate \bl{$N\in \{100,1000\}$} empirical samples $\{\bm \zeta^j\}_{j\in [N]}\in \Re_+^{I\times n}$ from a uniform distribution over a box $[1,10]^{I\times n}$. For each $l \in [n]$, we independently generated $c_l$ from the uniform distribution on the interval $[1, 10]$, while for each $i \in [I]$, we set $b^i := 50$. We tested these 10 random instances with risk parameter $\epsilon \in \{0.05,0.10\}$ and Wasserstein radius $\delta\in \{0.01,0.02\}$.

The numerical results \bl{with sample size $N=100$} are displayed in Table~\ref{tab:my-table}, where we use BigM Model, $\VaR$ Model, $\CVaR$ Model and ICCP Model denote exact formulation in Theorem~\ref{cor_reform_multi_drccp}, outer approximation in Corollary~\ref{cor_reform_multi_drccp_var}, $\CVaR$ approximation in Theorem~\ref{thm_cvar} and inner chance constrained programming approximation in Corollary~\ref{ref_math_thm_inner_2_cor1}, respectively. We also use ``Opt.Val" to denote the optimal value $v^*$, ``Value" to denote the best objective value output from an approximation model and ``Time" to denote the computational time in seconds. Additionally, since we can solve exact BigM Model to the optimality, we use GAP denote the optimality gap of an approximation model, which is computed as
\[\text{GAP}=\frac{|\text{Value}-\text{Opt.Val}|}{\text{Opt.Val}}.\]
We also let $\alpha^*$ denote the best $\alpha$ found in ICCP Model. In BigM Model \eqref{zset_P_W_I_1+_bigM}, we chose 
a lower bound of $\nu$ as $\underline{\nu}=1$. \bl{We chosen the big-M coefficients in BigM Model, $\VaR$ Model, and ICCP Model according to the remarks after Theorem~\ref{cor_reform_multi_drccp}, Corollary~\ref{cor_reform_multi_drccp_var}, and Corollary~\ref{ref_math_thm_inner_2_cor1}, respectively, where $\bm{L}=\bm0,\bm{U}=\e$.} From Table~\ref{tab:my-table}, we see that all the models can be solved to the optimality within 2 minutes, where BigM Model and ICCP Model often take the longest time to solve, and for each instance, $\CVaR$ Model can be solved within a second. This might be because (i) $\CVaR$ Model is a second order conic program and does not involve any binary variables; (ii) on the contrary, the BigM Model not only has binary variables but also involves the most number of auxiliary variables, while to solve ICCP Model, one needs to solve $\lceil N\epsilon\rceil$ regular chance constrained programs. In terms of approximation accuracy, we see that VaR Model is usually 2-3\% away from the true optimality, $\CVaR$ Model is 1-2\% away from the true optimality, while ICCP Model nearly finds the true optimal solution. This demonstrates that all of the proposed approximation models can find near-optimal solutions.

\bl{ The numerical results \bl{with sample size $N=1000$} are displayed in Table~\ref{tab:my-table11}, where similarly, we use BigM Model, $\VaR$ Model, $\CVaR$ Model and ICCP Model denote exact formulation in Theorem~\ref{cor_reform_multi_drccp}, outer approximation in Corollary~\ref{cor_reform_multi_drccp_var}, $\CVaR$ approximation in Theorem~\ref{thm_cvar} and inner chance constrained programming approximation in Corollary~\ref{ref_math_thm_inner_2_cor1}, respectively. We use ``UB" to denote the best upper bound found by BigM Model or $\VaR$ Model, ``LB" to denote the best lower bound found by BigM Model, $\CVaR$ Model, or ICCP Model, and ``Time" to denote the computational time in seconds. Additionally, since we cannot solve the BigM Model to optimality, we use GAP denote its optimality gap, which is computed as
\[\text{GAP}=\frac{|\text{UB}-\text{LB}|}{\text{LB}}.\]
To evaluate the effectiveness of approximation models, we use Improvement to denote the percentage of differences between the bounds of approximation models and bounds of BigM Model, i.e., for the $\VaR$ Model,
\[\text{Improvement}=1-\frac{\text{UB of $\VaR$ Model}}{\text{UB of BigM Model}},\]
while for the $\CVaR$ Model or ICCP Model, 
\[\text{Improvement}=\frac{\text{LB of Approximation Model}}{\text{LB of BigM Model}}-1,\]
where Approximation Model here is either $\CVaR$ Model or ICCP Model. We found that ICCP Model is difficult to solve these instances to optimality, and thus we chose a particular $\alpha=\frac{\epsilon}{2}$ in ICCP Model. Similarly, in BigM Model \eqref{zset_P_W_I_1+_bigM}, we chose 
a lower bound of $\nu$ as $\underline{\nu}=1$, and the big-M coefficients in BigM Model, $\VaR$ Model, and ICCP Model were computed according to the remarks after Theorem~\ref{cor_reform_multi_drccp}, Corollary~\ref{cor_reform_multi_drccp_var}, and Corollary~\ref{ref_math_thm_inner_2_cor1}, respectively, where $\bm{L}=\bm 0,\bm{U}=\e$. From Table~\ref{tab:my-table11}, we see that $\CVaR$ Model can be solved to optimality within 2 seconds, while all the other models cannot be solved within the time limit. In terms of approximation accuracy, we see that VaR Model consistently provides better upper bounds and can help close more 10\% optimality gap on average compared to BigM Model, $\CVaR$ Model often provides slightly better feasible solutions than BigM Model, while, ICCP Model yields the best feasible solutions. This demonstrates that all of the proposed approximation models are useful, to some extent, to improve the exact bigM model. In particular, $\VaR$ Model provides a better upper bound, which helps evaluate the solution quality more accurately, $\CVaR$ Model and ICCP Model often provide better feasible solutions. Also, we notice that mixed integer programs $\VaR$ Model and ICCP Model outperform BigM Model, which might be because (i) the BigM Model requires more auxiliary variables than $\VaR$ Model or ICCP Model; (ii) the naive big-M coefficients of the BigM Model are typically larger than the other two. }
In practice, it is worthy of trying all the BigM Model, ICCP Model, and CVaR Model first, then choose the best solution from three models and use the outer approximation- VaR Model to provide a numerical optimality guarantee on how good the solution quality is.

\begin{table}[htbp]\scriptsize
\caption{Numerical results of the exact formulation in Theorem~\ref{cor_reform_multi_drccp}, outer approximation in Corollary~\ref{cor_reform_multi_drccp_var}, $\CVaR$ approximation in Theorem~\ref{thm_cvar} and inner chance constrained programming approximation in Corollary~\ref{ref_math_thm_inner_2_cor1} on solving continuous DRMKP when sample size $N=1000$.}
\label{tab:my-table}
\begin{tabular}{ccc|cc|ccc|ccc|cccc}
\hline
\multirow{2}{*}{$\epsilon$} & \multirow{2}{*}{$\delta$} & \multirow{2}{*}{\begin{tabular}[c]{@{}c@{}}Insta-\\ nces\end{tabular}} & \multicolumn{2}{c|}{BigM Model} & \multicolumn{3}{c|}{$\VaR$ Model} & \multicolumn{3}{c|}{$\CVaR$ Model} & \multicolumn{4}{c}{ICCP Model} \\ \cline{4-15} 
 &  &  & Opt.Val & Time & Value & GAP & Time & Value & GAP & Time & Value & $\alpha^*$ & GAP & Time \\ \hline
\multirow{10}{*}{0.05} & \multirow{10}{*}{0.01} & 1 & 54.93 & 6.11 & 56.37 & 2.62\% & 3.37 & 54.30 & 1.14\% & 0.06 & 54.93 & 0.03 & 0.00\% & 9.43 \\
 &  & 2 & 47.69 & 5.24 & 48.79 & 2.29\% & 2.04 & 47.16 & 1.11\% & 0.05 & 47.69 & 0.03 & 0.00\% & 7.90 \\
 &  & 3 & 50.73 & 4.44 & 51.43 & 1.38\% & 4.43 & 50.38 & 0.70\% & 0.05 & 50.73 & 0.02 & 0.00\% & 8.64 \\
 &  & 4 & 53.97 & 3.61 & 54.98 & 1.87\% & 4.75 & 52.72 & 2.32\% & 0.06 & 53.97 & 0.03 & 0.00\% & 8.16 \\
 &  & 5 & 54.96 & 6.99 & 56.44 & 2.68\% & 4.20 & 52.88 & 3.79\% & 0.05 & 54.96 & 0.03 & 0.00\% & 7.42 \\
 &  & 6 & 56.03 & 6.46 & 57.40 & 2.44\% & 2.64 & 54.97 & 1.89\% & 0.05 & 56.03 & 0.03 & 0.00\% & 6.35 \\
 &  & 7 & 54.17 & 6.69 & 55.04 & 1.62\% & 3.68 & 53.26 & 1.67\% & 0.05 & 54.12 & 0.02 & 0.08\% & 7.92 \\
 &  & 8 & 55.40 & 5.81 & 56.55 & 2.09\% & 3.19 & 54.15 & 2.26\% & 0.05 & 55.40 & 0.03 & 0.00\% & 6.86 \\
 &  & 9 & 57.63 & 4.91 & 58.95 & 2.29\% & 4.20 & 57.07 & 0.96\% & 0.05 & 57.62 & 0.02 & 0.02\% & 10.80 \\
 &  & 10 & 56.31 & 4.34 & 57.15 & 1.50\% & 4.71 & 55.95 & 0.63\% & 0.06 & 56.31 & 0.02 & 0.00\% & 8.62 \\ \hline
\multicolumn{3}{c|}{Average} &  & 5.46 &  & 2.08\% & 3.72 &  & 1.65\% & 0.05 &  &  & 0.01\% & 8.21 \\ \hline\hline
\multirow{10}{*}{0.05} & \multirow{10}{*}{0.02} & 1 & 53.97 & 3.94 & 55.92 & 3.63\% & 3.27 & 53.83 & 0.24\% & 0.05 & 53.94 & 0.02 & 0.05\% & 9.95 \\
 &  & 2 & 47.05 & 3.63 & 48.42 & 2.92\% & 3.20 & 46.79 & 0.53\% & 0.04 & 47.04 & 0.02 & 0.01\% & 8.64 \\
 &  & 3 & 50.12 & 5.26 & 51.02 & 1.79\% & 4.48 & 49.96 & 0.33\% & 0.05 & 50.11 & 0.01 & 0.03\% & 8.88 \\
 &  & 4 & 52.98 & 5.14 & 54.49 & 2.84\% & 4.83 & 52.28 & 1.33\% & 0.06 & 52.98 & 0.02 & 0.00\% & 9.41 \\
 &  & 5 & 54.10 & 3.76 & 55.95 & 3.41\% & 3.67 & 52.44 & 3.07\% & 0.05 & 54.05 & 0.02 & 0.09\% & 9.55 \\
 &  & 6 & 55.16 & 6.02 & 56.90 & 3.16\% & 3.33 & 54.52 & 1.17\% & 0.05 & 55.14 & 0.02 & 0.04\% & 7.58 \\
 &  & 7 & 53.41 & 3.91 & 54.55 & 2.13\% & 3.81 & 52.83 & 1.08\% & 0.05 & 53.38 & 0.02 & 0.06\% & 7.59 \\
 &  & 8 & 54.47 & 2.77 & 56.09 & 2.98\% & 3.34 & 53.71 & 1.39\% & 0.06 & 54.43 & 0.02 & 0.07\% & 6.63 \\
 &  & 9 & 56.85 & 3.40 & 58.44 & 2.79\% & 4.00 & 56.59 & 0.46\% & 0.05 & 56.84 & 0.01 & 0.02\% & 9.39 \\
 &  & 10 & 55.65 & 5.47 & 56.71 & 1.90\% & 4.90 & 55.53 & 0.22\% & 0.06 & 55.65 & 0.01 & 0.00\% & 9.29 \\ \hline
\multicolumn{3}{c|}{Average} &  & 4.33 &  & 2.76\% & 3.88 &  & 0.98\% & 0.05 &  &  & 0.04\% & 8.69 \\ \hline\hline
\multirow{10}{*}{0.1} & \multirow{10}{*}{0.01} & 1 & 56.47 & 25.78 & 57.71 & 2.19\% & 10.01 & 55.14 & 2.36\% & 0.05 & 56.47 & 0.06 & 0.00\% & 35.61 \\
 &  & 2 & 48.82 & 66.63 & 49.87 & 2.16\% & 6.51 & 48.00 & 1.68\% & 0.06 & 48.79 & 0.06 & 0.06\% & 25.94 \\
 &  & 3 & 51.58 & 102.52 & 52.56 & 1.89\% & 10.35 & 50.93 & 1.26\% & 0.06 & 51.58 & 0.07 & 0.00\% & 54.93 \\
 &  & 4 & 55.28 & 15.06 & 56.20 & 1.66\% & 8.26 & 53.97 & 2.37\% & 0.05 & 55.28 & 0.06 & 0.00\% & 35.40 \\
 &  & 5 & 56.94 & 22.07 & 58.51 & 2.75\% & 4.82 & 54.28 & 4.68\% & 0.06 & 56.94 & 0.07 & 0.00\% & 31.10 \\
 &  & 6 & 57.50 & 23.31 & 58.94 & 2.51\% & 9.04 & 55.92 & 2.74\% & 0.07 & 57.50 & 0.06 & 0.00\% & 29.69 \\
 &  & 7 & 55.21 & 24.96 & 56.51 & 2.35\% & 5.15 & 54.24 & 1.77\% & 0.06 & 55.19 & 0.06 & 0.04\% & 27.51 \\
 &  & 8 & 56.64 & 15.43 & 57.96 & 2.33\% & 3.80 & 55.42 & 2.15\% & 0.06 & 56.60 & 0.06 & 0.08\% & 29.52 \\
 &  & 9 & 59.18 & 23.14 & 60.47 & 2.19\% & 8.79 & 58.01 & 1.98\% & 0.08 & 59.14 & 0.07 & 0.07\% & 46.42 \\
 &  & 10 & 57.20 & 29.34 & 58.02 & 1.44\% & 10.08 & 56.50 & 1.21\% & 0.07 & 57.19 & 0.07 & 0.00\% & 48.08 \\ \hline
\multicolumn{3}{c|}{Average} &  & 34.83 &  & 2.15\% & 7.68 &  & 2.22\% & 0.06 &  &  & 0.03\% & 36.42 \\ \hline\hline
\multirow{10}{*}{0.1} & \multirow{10}{*}{0.02} & 1 & 55.93 & 77.63 & 57.45 & 2.72\% & 9.50 & 54.89 & 1.85\% & 0.05 & 55.92 & 0.05 & 0.01\% & 36.73 \\
 &  & 2 & 48.47 & 20.09 & 49.66 & 2.47\% & 2.77 & 47.82 & 1.34\% & 0.06 & 48.42 & 0.05 & 0.09\% & 27.42 \\
 &  & 3 & 51.14 & 110.17 & 52.34 & 2.35\% & 14.61 & 50.72 & 0.81\% & 0.07 & 51.06 & 0.04 & 0.14\% & 54.71 \\
 &  & 4 & 54.68 & 73.17 & 55.96 & 2.35\% & 12.99 & 53.74 & 1.71\% & 0.06 & 54.67 & 0.06 & 0.01\% & 43.06 \\
 &  & 5 & 56.11 & 16.04 & 58.25 & 3.81\% & 3.98 & 54.05 & 3.67\% & 0.06 & 56.11 & 0.07 & 0.00\% & 34.82 \\
 &  & 6 & 56.93 & 18.81 & 58.66 & 3.05\% & 3.61 & 55.68 & 2.19\% & 0.06 & 56.90 & 0.05 & 0.05\% & 33.87 \\
 &  & 7 & 54.67 & 37.46 & 56.26 & 2.90\% & 6.57 & 54.00 & 1.22\% & 0.05 & 54.61 & 0.06 & 0.12\% & 33.40 \\
 &  & 8 & 56.15 & 15.48 & 57.71 & 2.77\% & 4.54 & 55.20 & 1.70\% & 0.05 & 56.09 & 0.05 & 0.11\% & 26.77 \\
 &  & 9 & 58.51 & 18.82 & 60.21 & 2.91\% & 10.79 & 57.76 & 1.28\% & 0.06 & 58.48 & 0.04 & 0.05\% & 47.85 \\
 &  & 10 & 56.76 & 33.72 & 57.80 & 1.84\% & 12.03 & 56.29 & 0.83\% & 0.07 & 56.71 & 0.05 & 0.08\% & 44.74 \\ \hline
\multicolumn{3}{c|}{Average} &  & 42.14 &  & 2.72\% & 8.14 &  & 1.66\% & 0.06 &  &  & 0.07\% & 38.34 \\ \hline
\end{tabular}
\end{table}

\begin{table}[htbp]\scriptsize\centering
\caption{Numerical results of the exact formulation in Theorem~\ref{cor_reform_multi_drccp}, outer approximation in Corollary~\ref{cor_reform_multi_drccp_var}, $\CVaR$ approximation in Theorem~\ref{thm_cvar} and inner chance constrained programming approximation in Corollary~\ref{ref_math_thm_inner_2_cor1} on solving continuous DRMKP when sample size $N=1000$.}
\label{tab:my-table11}\setlength{\tabcolsep}{3pt}
\begin{tabular}{ccc|cccc|ccc|ccc|cccc}
\hline
\multirow{2}{*}{$\epsilon$} & \multirow{2}{*}{$\delta$} & \multirow{2}{*}{\begin{tabular}[c]{@{}c@{}}Insta-\\ nces\end{tabular}} & \multicolumn{4}{c|}{BigM Model} & \multicolumn{3}{c|}{$\VaR$ Model} & \multicolumn{3}{c|}{$\CVaR$ Model} & \multicolumn{4}{c}{ICCP Model} \\ \cline{4-17} 
 &  &  & UB & LB & GAP & Time & UB &  \begin{tabular}[c]{@{}c@{}}Improv-\\ ment\end{tabular} & Time & LB &  \begin{tabular}[c]{@{}c@{}}Improv-\\ ment\end{tabular} & Time & LB & $\alpha$ & \begin{tabular}[c]{@{}c@{}}Improv-\\ ment\end{tabular} & Time \\\hline
\multirow{10}{*}{0.05} & \multirow{10}{*}{0.01}  & 1 & 66.45 & 53.02 & 25.32\% & 3600 & 57.71 & 13.15\% & 3600 & 52.94 & -0.15\% & 0.72 & 53.47 & 0.025 & 0.84\% & 3600 \\
 &  & 2 & 64.48 & 52.18 & 23.57\% & 3600 & 57.07 & 11.50\% & 3600 & 52.51 & 0.62\% & 0.61 & 52.57 & 0.025 & 0.74\% & 3600 \\
 &  & 3 & 69.55 & 54.02 & 28.75\% & 3600 & 62.21 & 10.54\% & 3600 & 54.45 & 0.81\% & 1.00 & 55.13 & 0.025 & 2.06\% & 3600 \\
 &  & 4 & 68.02 & 53.91 & 26.18\% & 3600 & 61.60 & 9.45\% & 3600 & 54.23 & 0.58\% & 1.01 & 54.60 & 0.025 & 1.28\% & 3600 \\
 &  & 5 & 68.39 & 56.75 & 20.52\% & 3600 & 61.25 & 10.44\% & 3600 & 56.80 & 0.09\% & 0.65 & 57.12 & 0.025 & 0.65\% & 3600 \\
 &  & 6 & 69.66 & 56.15 & 24.05\% & 3600 & 60.74 & 12.81\% & 3600 & 56.32 & 0.29\% & 0.84 & 56.00 & 0.025 & -0.28\% & 3600 \\
 &  & 7 & 74.45 & 57.95 & 28.47\% & 3600 & 66.33 & 10.91\% & 3600 & 58.11 & 0.29\% & 1.08 & 58.65 & 0.025 & 1.21\% & 3600 \\
 &  & 8 & 74.91 & 57.42 & 30.45\% & 3600 & 66.42 & 11.33\% & 3600 & 57.86 & 0.75\% & 0.93 & 58.04 & 0.025 & 1.07\% & 3600 \\
 &  & 9 & 65.84 & 51.64 & 27.49\% & 3600 & 56.72 & 13.85\% & 3600 & 51.85 & 0.41\% & 0.51 & 52.16 & 0.025 & 1.01\% & 3600 \\
 &  & 10 & 66.26 & 50.94 & 30.06\% & 3600 & 56.11 & 15.32\% & 3600 & 51.43 & 0.96\% & 0.62 & 51.37 & 0.025 & 0.85\% & 3600 \\\hline
\multicolumn{3}{c|}{Average}  &  &  & 26.49\% & 3600 &  & 11.93\% & 3600 &  & 0.46\% & 0.80 &  &  & 0.94\% & 3600 \\\hline\hline
\multirow{10}{*}{0.05} & \multirow{10}{*}{0.02} & 1 & 72.24 & 53.22 & 35.74\% & 3600 & 60.48 & 16.27\% & 3600 & 53.10 & -0.22\% & 0.90 & 53.54 & 0.05 & 0.62\% & 3600 \\
 &  & 2 & 68.90 & 52.53 & 31.15\% & 3600 & 60.68 & 11.94\% & 3600 & 52.88 & 0.65\% & 1.18 & 53.10 & 0.05 & 1.07\% & 3600 \\
 &  & 3 & 57.60 & 46.77 & 23.14\% & 3600 & 51.53 & 10.53\% & 3600 & 46.94 & 0.35\% & 0.95 & 47.16 & 0.05 & 0.83\% & 3600 \\
 &  & 4 & 57.05 & 46.13 & 23.67\% & 3600 & 52.25 & 8.40\% & 3600 & 46.56 & 0.93\% & 0.82 & 46.41 & 0.05 & 0.61\% & 3600 \\
 &  & 5 & 60.61 & 48.11 & 25.99\% & 3600 & 54.86 & 9.49\% & 3600 & 48.12 & 0.02\% & 1.13 & 48.61 & 0.05 & 1.04\% & 3600 \\
 &  & 6 & 62.30 & 47.24 & 31.89\% & 3600 & 55.74 & 10.53\% & 3600 & 47.92 & 1.45\% & 1.49 & 48.00 & 0.05 & 1.61\% & 3600 \\
 &  & 7 & 61.50 & 48.76 & 26.12\% & 3600 & 53.23 & 13.44\% & 3600 & 49.11 & 0.72\% & 0.51 & 49.44 & 0.05 & 1.39\% & 3600 \\
 &  & 8 & 60.38 & 48.39 & 24.77\% & 3600 & 52.77 & 12.59\% & 3600 & 48.70 & 0.64\% & 0.57 & 48.62 & 0.05 & 0.46\% & 3600 \\
 &  & 9 & 65.20 & 49.94 & 30.56\% & 3600 & 57.51 & 11.80\% & 3600 & 50.36 & 0.84\% & 1.05 & 50.63 & 0.05 & 1.39\% & 3600 \\
 &  & 10 & 64.73 & 49.71 & 30.22\% & 3600 & 57.68 & 10.90\% & 3600 & 50.15 & 0.88\% & 1.08 & 50.28 & 0.05 & 1.16\% & 3600 \\\hline
\multicolumn{3}{c|}{Average}  &  &  & 28.33\% & 3600 &  & 11.59\% & 3600 &  & 0.63\% & 0.97 &  &  & 1.02\% & 3600 \\\hline\hline
\multirow{10}{*}{0.1} & \multirow{10}{*}{0.01} & 1 & 68.21 & 51.92 & 31.39\% & 3600 & 57.01 & 16.43\% & 3600 & 51.99 & 0.14\% & 0.62 & 52.38 & 0.025 & 0.88\% & 3600 \\
 &  & 2 & 67.75 & 51.22 & 32.26\% & 3600 & 57.75 & 14.75\% & 3600 & 51.57 & 0.68\% & 0.66 & 51.55 & 0.025 & 0.63\% & 3600 \\
 &  & 3 & 70.55 & 53.00 & 33.11\% & 3600 & 62.53 & 11.37\% & 3600 & 53.24 & 0.44\% & 0.78 & 53.93 & 0.025 & 1.75\% & 3600 \\
 &  & 4 & 70.47 & 52.49 & 34.26\% & 3600 & 60.78 & 13.75\% & 3600 & 53.02 & 1.01\% & 0.87 & 53.35 & 0.025 & 1.64\% & 3600 \\
 &  & 5 & 64.24 & 51.29 & 25.27\% & 3600 & 55.59 & 13.47\% & 3600 & 51.38 & 0.18\% & 0.53 & 51.78 & 0.025 & 0.96\% & 3600 \\
 &  & 6 & 63.74 & 50.69 & 25.75\% & 3600 & 54.92 & 13.84\% & 3600 & 50.95 & 0.52\% & 0.61 & 50.90 & 0.025 & 0.42\% & 3600 \\
 &  & 7 & 66.93 & 52.90 & 26.52\% & 3600 & 58.97 & 11.89\% & 3600 & 52.75 & -0.29\% & 1.31 & 53.35 & 0.025 & 0.84\% & 3600 \\
 &  & 8 & 66.56 & 52.42 & 26.97\% & 3600 & 59.37 & 10.80\% & 3600 & 52.53 & 0.20\% & 1.14 & 52.84 & 0.025 & 0.80\% & 3600 \\
 &  & 9 & 69.74 & 56.43 & 23.59\% & 3600 & 61.97 & 11.15\% & 3600 & 56.23 & -0.35\% & 0.75 & 56.83 & 0.025 & 0.71\% & 3600 \\
 &  & 10 & 67.45 & 55.72 & 21.05\% & 3600 & 60.87 & 9.76\% & 3600 & 55.78 & 0.11\% & 0.66 & 55.94 & 0.025 & 0.39\% & 3600 \\\hline
\multicolumn{3}{c|}{Average}  &  &  & 28.02\% & 3600 &  & 12.72\% & 3600 &  & 0.26\% & 0.79 &  &  & 0.90\% & 3600 \\\hline\hline
\multirow{10}{*}{0.1} & \multirow{10}{*}{0.02} & 1 & 75.06 & 58.30 & 28.75\% & 3600 & 67.10 & 10.60\% & 3600 & 57.67 & -1.08\% & 0.93 & 58.39 & 0.05 & 0.15\% & 3600 \\
 &  & 2 & 75.79 & 57.11 & 32.72\% & 3600 & 65.20 & 13.97\% & 3600 & 57.43 & 0.56\% & 0.96 & 57.87 & 0.05 & 1.34\% & 3600 \\
 &  & 3 & 71.66 & 55.49 & 29.13\% & 3600 & 59.52 & 16.93\% & 3600 & 55.21 & -0.51\% & 0.50 & 55.69 & 0.05 & 0.36\% & 3600 \\
 &  & 4 & 67.30 & 54.31 & 23.92\% & 3600 & 59.01 & 12.31\% & 3600 & 54.74 & 0.80\% & 0.56 & 54.78 & 0.05 & 0.87\% & 3600 \\
 &  & 5 & 73.91 & 57.35 & 28.88\% & 3600 & 65.21 & 11.77\% & 3600 & 56.72 & -1.09\% & 0.89 & 57.37 & 0.05 & 0.04\% & 3600 \\
 &  & 6 & 72.05 & 56.71 & 27.04\% & 3600 & 64.56 & 10.39\% & 3600 & 56.48 & -0.41\% & 1.00 & 56.88 & 0.05 & 0.29\% & 3600 \\
 &  & 7 & 70.33 & 56.87 & 23.67\% & 3600 & 60.52 & 13.94\% & 3600 & 56.25 & -1.09\% & 0.59 & 57.03 & 0.05 & 0.29\% & 3600 \\
 &  & 8 & 66.69 & 55.88 & 19.34\% & 3600 & 60.25 & 9.65\% & 3600 & 55.76 & -0.22\% & 0.77 & 56.08 & 0.05 & 0.35\% & 3600 \\
 &  & 9 & 73.51 & 58.36 & 25.94\% & 3600 & 65.38 & 11.05\% & 3600 & 57.99 & -0.64\% & 0.89 & 58.77 & 0.05 & 0.69\% & 3600 \\
 &  & 10 & 73.98 & 57.51 & 28.65\% & 3600 & 66.01 & 10.78\% & 3600 & 57.73 & 0.38\% & 1.16 & 58.18 & 0.05 & 1.17\% & 3600 \\\hline
\multicolumn{3}{c|}{Average}  &  &  & 26.80\% & 3600 &  & 12.14\% & 3600 &  & -0.33\% & 0.83 &  &  & 0.55\% & 3600\\\hline
\end{tabular}
\end{table}

\subsection{Choosing a Wasserstein Radius using Cross Validation}

In this subsection, we use continuous DRMKP (i.e., $S=[0,1]^n$ in \eqref{eq_drccp_con_mdk}) to numerically demonstrate how to use cross validation to choose a proper Wasserstein radius $\delta$ and also test the effects of the correlation of the random vectors $(\tilde{\rxi}_1,\ldots,\tilde{\rxi}_I)$. We suppose that $\epsilon=0.05$, $n=20$ and $I=10$, and $\tilde{\rxi}_i=\rho\bar{\rxi}+(1-\rho)\hat{\rxi}_i$ for each $i\in [I]$, where $\bar{\rxi},\hat{\rxi}_1,\ldots,\hat{\rxi}_I $ are independent uniform random vectors over the box $[1,10]^{I\times n}$, and $\rho\in [0,1]$. Clearly, as $\rho$ grows, the correlation among random vectors $(\tilde{\rxi}_1,\ldots,\tilde{\rxi}_I)$ increases. Our numerical instances were generated as follows. We first generated $N=100$ i.i.d. samples $\{(\hat{\bm \zeta}_1^j,\ldots,\hat{\bm \zeta}_I^j)\}_{j\in [N]}\in \Re_+^{I\times n}$ from a uniform distribution over a box $[1,10]^{I\times n}$, and $N=100$ i.i.d. samples $\{\bar{\bm \zeta}^j\}_{j\in [N]}\in \Re_+^{n}$ from a uniform distribution over a box $[1,10]^{n}$, where $n=20$ and $I=10$. Next, we constructed 11 instances with $N=100$ empirical samples as 
$$\bm \zeta_i^j=\rho \bar{\bm \zeta}^j+(1-\rho)\hat{\bm \zeta}_i^j$$
for each $i\in [I]$ and $j\in [N]$, and $\rho\in\{0,0.1,\ldots,1\}$. 
For each $l \in [n]$, we independently generated $c_l$ from the uniform distribution on the interval $[1, 10]$, while for each $i \in [I]$, we set $b^i := 50$. We also suppose that the possible Wasserstein radii are from $\delta\in \{0.01,0.02,\ldots,0.1\}$.

The cross validation procedure was done in the following manner: (i) for each $\delta\in \{0.01,0.02,\ldots,0.1\}$, we solved continuous DRMKP to the optimality using exact formulation in Proposition~\ref{thm_reform_joint_drccp}; (ii) we generated $10^4$ samples of the random vectors $(\tilde{\rxi}_1,\ldots,\tilde{\rxi}_I)$ and used these samples to estimate the violation probability of uncertain constraints with respect to the optimal solution found at step (i). We repeated this procedure 10 times and output the 90-percentile of these estimated violation probabilities, denoted as 90-percentile violation; and (iii) we chose the best Wasserstein radius $\delta^*$ as the smallest $\delta\in \{0.01,0.02,\ldots,0.1\}$ such that its 90-percentile violation is below the target violation, i.e., less than or equal to $\epsilon=0.05$, \bl{which implies that approximately with probability at least 0.9, the DRMKP solution will be feasible to its regular CCP Model}.

For the comparison purpose, we also solve a regular chance constrained programming counterpart of DRMKP \eqref{eq_drccp_con_mdk} with respect to the empirical samples $\{{\bm \zeta}^j\}_{j\in [N]}$ and used the same procedure to compute its 90-percentile violation. The numerical results are displayed in Table~\ref{tab:my-table2}, where we use ``CCP Model" to denote the chance constrained programming counterpart of DRMKP, and ``Opt.Val" to denote the optimal value of a corresponding model. 

From Table~\ref{tab:my-table}, we see that for the optimal solutions from CCP Model often have much higher probability of violating the uncertain constraints than the target risk parameter $\epsilon=0.05$. On the other hand, for DRMKP Model, by choosing Wasserstein radius properly, its probability of violating the uncertain constraints is often smaller than the target risk parameter, and its optimal value is often very close to that of CCP Model. This demonstrates the robustness and accuracy of the proposed DRMKP Model. We also note that when $\rho$ grows, i.e., the correlation between among random vectors $(\tilde{\rxi}_1,\ldots,\tilde{\rxi}_I)$ increases, the best Wasserstein radius $\delta^*$ does not tend to decrease or increase. This result demonstrates that Assumption~\ref{assume_A2} does not cause too much over-conservatism of the proposed DRCCP models. In fact, we note that the best Wasserstein radius $\delta^*$ is positively correlated with the 90-percentile violation of CCP Model, i.e., $\delta^*$ tends to be bigger if CCP Model has a larger 90-percentile violation value. In practice, we suggest solving regular CCP Model first and then choose a proper range of $\delta$ for the cross validation. Finally, if the cross validation takes too much time due to the difficulty of solving MILPs, then we can reduce the running time via warm start. That is, we suggest solving the cross validation instances in the descending order of possible $\delta$ values, and when solving a cross validation instance, since the optimal solution from previous instance is feasible to the current one, thus, we can input this solution to the solver as a starting point.
\begin{table}[htbp]\scriptsize
\centering
\caption{Illustration of choosing a Wasserstein radius using cross validation. \bl{Note: $\delta$ is chosen from $\{0.01,0.02,\ldots,0.1\}$}}
\label{tab:my-table2}
\begin{tabular}{cccc|cc|c}
\hline
\multicolumn{1}{c|}{\multirow{2}{*}{$\rho$}} & \multicolumn{3}{c|}{DRMKP Model} & \multicolumn{2}{c|}{CCP Model} & \multirow{2}{*}{\begin{tabular}[c]{@{}c@{}}Target\\ Violation\\  ($\epsilon$)\end{tabular}} \\ \cline{2-6}
\multicolumn{1}{c|}{} & $\delta^*$ & Opt.Val & \begin{tabular}[c]{@{}c@{}}90-Percentile\\ Violation\end{tabular} & Opt.Val & \begin{tabular}[c]{@{}c@{}}90-Percentile \\ Violation\end{tabular} &  \\ \hline
\multicolumn{1}{c|}{0} & 0.03 & 53.76 & 0.04154 & 56.99 & 0.13461 & \multirow{11}{*}{0.05} \\
\multicolumn{1}{c|}{0.1} & 0.02 & 50.06 &  0.04431 & 52.67 &0.08698 &  \\
\multicolumn{1}{c|}{0.2} & 0.03 & 52.37 &0.03133
 & 55.11 & 0.15273 &  \\
\multicolumn{1}{c|}{0.3} & 0.01 & 56.94 & 0.03905 & 58.33 & 0.09624 &  \\
\multicolumn{1}{c|}{0.4} & 0.02 & 53.38 & 0.02801 & 55.89 & 0.12054 &  \\
\multicolumn{1}{c|}{0.5} & 0.02 & 50.25 & 0.03249 & 52.13 & 0.09629 &  \\
\multicolumn{1}{c|}{0.6} & 0.01 & 59.38 & 0.04671 & 60.98 & 0.08015 &  \\
\multicolumn{1}{c|}{0.7} & 0.03 & 54.60 & 0.04742 & 57.77 & 0.12871 &  \\
\multicolumn{1}{c|}{0.8} & 0.03 & 62.51 & 0.04678 & 66.39 & 0.11837 &  \\
\multicolumn{1}{c|}{0.9} & 0.03 & 52.82 & 0.0364 & 56.90 & 0.13221 &  \\
\multicolumn{1}{c|}{1} & 0.02 & 59.51 & 0.03998 & 62.09 & 0.09496 &  \\ \hline
\end{tabular}
\end{table}

\subsection{Binary DRMKP: Strength of Big-M Free Formulation}

In this subsection, we present a numerical study to compare the big-M formulation in Theorem~\ref{cor_reform_multi_drccp} with big-M free formulation in Corollary~\ref{cor_reform_joint_drccp_binary} on solving binary DRMKP (i.e., $S=\{0,1\}^n$ in \eqref{eq_drccp_con_mdk}). 
To test the proposed formulations, we generated 10 random instances with $n=20$ and $I=10$, indexed by $\{1,2,\ldots,10\}$. For each instance, we generated $N=1000$ empirical samples $\{\bm \zeta^j\}_{j\in [N]}\in \Re_+^{I\times n}$ from a uniform distribution over a box $[1,10]^{I\times n}$. For each $l \in [n]$, we independently generated $c_l$ from the uniform distribution on the interval $[1, 10]$, while for each $i \in [I]$, we set $b^i := 100$. We tested these 10 random instances with risk parameter $\epsilon \in \{0.05,0.10\}$ and Wasserstein radius $\delta\in \{0.1,0.2\}$. Also, in BigM Model \eqref{zset_P_W_I_1+_bigM}, we chose 
$$M_j=\max_{i\in [I]}\max\left\{b^i,\sum_{l\in [n]}\zeta_{il}^j-b^i\right\},\forall j\in [N],$$
and a lower bound of $\nu$ as $\underline{\nu}=1$. \bl{In the branch and cut implementation described in the end of Section~\ref{sec_DRMKP}, each time we added $\kappa=10$ EPI inequalities.}

The results are displayed in Table~\ref{table_num_results_drccp}. 
We use BigM Model and BigM-free Model to denote the big-M formulation in Theorem~\ref{cor_reform_multi_drccp} and big-M free formulation in Corollary~\ref{cor_reform_joint_drccp_binary}, respectively. In addition, we use UB, LB, GAP, Opt.Val and Time to denote the best upper bound, the best lower bound, optimality gap, the optimal objective value, and the total running time in seconds, respectively. 

\begin{table}[htbp]\scriptsize
	~\vspace{-10pt}
	\centering
	\caption{Numerical comparison of big-M formulation in Theorem~\ref{cor_reform_multi_drccp} and big-M free formulation in Corollary~
\ref{cor_reform_joint_drccp_binary} on solving binary DRMKP}
	\label{table_num_results_drccp}
\begin{tabular}{ccccc|cccc|cc}
	\hline
	\multirow{2}{*}{$\epsilon$} &\multirow{2}{*}{$\delta$}  & \multirow{2}{*}{Instances} & \multirow{2}{*}{$n$} & \multirow{2}{*}{$I$} & \multicolumn{4}{c|}{BigM Model} & \multicolumn{2}{c}{BigM-free Model} \\ \cline{6-11}
	& &  &  &  & UB&LB & Time &GAP & Opt.Val & Time \\ \hline
\multirow{10}{*}{0.05} & \multirow{10}{*}{0.1} & 1 & 20 & 10 & 93 & 86 & 3600.0 & 7.5\% & 89 & 49.3 \\
 &  & 2 & 20 & 10 & 97 & 90 & 3600.0 & 7.2\% & 95 & 30.6 \\
 &  & 3 & 20 & 10 & 95 & 84 & 3600.0 & 11.6\% & 90 & 387.0 \\
 &  & 4 & 20 & 10 & 84 & 74 & 3600.0 & 11.9\% & 78 & 275.7 \\
 &  & 5 & 20 & 10 & 87 & 81 & 3600.0 & 6.9\% & 82 & 140.4 \\
 &  & 6 & 20 & 10 & 97 & 85 & 3600.0 & 12.4\% & 88 & 972.5 \\
 &  & 7 & 20 & 10 & 89 & 75 & 3600.0 & 15.7\% & 84 & 169.6 \\
 &  & 8 & 20 & 10 & 100 & 88 & 3600.0 & 12.0\% & 96 & 80.5 \\
 &  & 9 & 20 & 10 & 96 & 78 & 3600.0 & 18.8\% & 92 & 59.3 \\
 &  & 10 & 20 & 10 & 93 & 93 & 3542.7 & 0.0\% & 93 & 18.2 \\\hline
\multicolumn{5}{c|}{Average} &  &  & 3594.3 & 10.4\% &  & 218.3 \\\hline\hline
\multirow{10}{*}{0.1} & \multirow{10}{*}{0.1} & 1 & 20 & 10 & 100 & NA & 3600.0 & NA & 92 & 172.9 \\
 &  & 2 & 20 & 10 & 106 & NA & 3600.0 & NA & 99 & 164.0 \\
 &  & 3 & 20 & 10 & 105 & 87 & 3600.0 & 17.1\% & 93 & 569.1 \\
 &  & 4 & 20 & 10 & 92 & 67 & 3600.0 & 27.2\% & 82 & 600.5 \\
 &  & 5 & 20 & 10 & 95 & NA & 3600.0 & NA & 86 & 332.0 \\
 &  & 6 & 20 & 10 & 109 & NA & 3600.0 & NA & 94 & 1852.4 \\
 &  & 7 & 20 & 10 & 96 & NA & 3600.0 & NA & 88 & 279.8 \\
 &  & 8 & 20 & 10 & 108 & 82 & 3600.0 & 24.1\% & 100 & 133.2 \\
 &  & 9 & 20 & 10 & 102 & NA & 3600.0 & NA & 94 & 389.3 \\
 &  & 10 & 20 & 10 & 103 & 96 & 3600.0 & 6.8\% & 96 & 149.7 \\\hline
\multicolumn{5}{c|}{Average} &  &  & 3600.0 & 18.8\% &  & 464.3 \\\hline\hline
\multirow{10}{*}{0.05} & \multirow{10}{*}{0.2} & 1 & 20 & 10 & 87 & 87 & 665.8 & 0.0\% & 87 & 8.5 \\
 &  & 2 & 20 & 10 & 88 & 88 & 2473.2 & 0.0\% & 88 & 19.3 \\
 &  & 3 & 20 & 10 & 86 & 86 & 1391.3 & 0.0\% & 86 & 70.4 \\
 &  & 4 & 20 & 10 & 74 & 74 & 2881.7 & 0.0\% & 74 & 102.5 \\
 &  & 5 & 20 & 10 & 78 & 78 & 1553.5 & 0.0\% & 78 & 26.9 \\
 &  & 6 & 20 & 10 & 86 & 86 & 2776.2 & 0.0\% & 86 & 442.7 \\
 &  & 7 & 20 & 10 & 83 & 83 & 1413.9 & 0.0\% & 83 & 17.1 \\
 &  & 8 & 20 & 10 & 92 & 92 & 297.7 & 0.0\% & 92 & 21.0 \\
 &  & 9 & 20 & 10 & 90 & 90 & 148.5 & 0.0\% & 90 & 14.6 \\
 &  & 10 & 20 & 10 & 90 & 90 & 1074.2 & 0.0\% & 90 & 8.9 \\\hline
\multicolumn{5}{c|}{Average} &  &  & 1467.6 & 0.0\% &  & 73.2 \\\hline\hline
\multirow{10}{*}{0.1} & \multirow{10}{*}{0.2} & 1 & 20 & 10 & 96 & 85 & 3600.0 & 11.5\% & 92 & 34.3 \\
 &  & 2 & 20 & 10 & 103 & 88 & 3600.0 & 14.6\% & 99 & 16.5 \\
 &  & 3 & 20 & 10 & 98 & 93 & 3600.0 & 5.1\% & 93 & 175.4 \\
 &  & 4 & 20 & 10 & 86 & 82 & 3600.0 & 4.7\% & 82 & 243.5 \\
 &  & 5 & 20 & 10 & 90 & NA & 3600.0 & NA & 86 & 84.7 \\
 &  & 6 & 20 & 10 & 101 & 81 & 3600.0 & 19.8\% & 94 & 524.6 \\
 &  & 7 & 20 & 10 & 90 & 88 & 3600.0 & 2.2\% & 88 & 93.1 \\
 &  & 8 & 20 & 10 & 103 & NA & 3600.0 & NA & 100 & 53.4 \\
 &  & 9 & 20 & 10 & 97 & 94 & 3600.0 & 3.1\% & 94 & 75.5 \\
 &  & 10 & 20 & 10 & 99 & 89 & 3600.0 & 10.1\% & 96 & 14.1 \\\hline
\multicolumn{5}{c|}{Average} &  &  & 3600.0 & 8.9\% &  & 131.5\\\hline
\end{tabular}
\\ 
$^{*}$ The NA represents that no feasible solution has been found within the time limit
\end{table}

From Table~\ref{table_num_results_drccp}, we observe that the overall running time of BigM-free Model significantly outperforms that of BigM Model, i.e., almost all of the instances of BigM-free Model can be solved within 10 minutes, while the majority of the instances of BigM Model reach the time limit.
The main reasons are two-fold: (i) BigM Model involves $\mathcal{O}(N+n)$ binary variables and $\mathcal{O}(N\times I)$ continuous variables, while BigM-free Model only involves $\mathcal{O}(n)$ binary variables and $\mathcal{O}(N)$ continuous variables; and (ii) BigM Model contains big-M coefficients, while BigM-free Model does not. We also observe that, as the risk parameter $\epsilon$ increases or Wasserstein radius $\delta$ decreases, both formulations take longer time to solve, but BigM-free Model still significantly outperforms BigM Model. These results demonstrate the effectiveness of our proposed BigM-free Model.

}
\section{Conclusion}\label{sec_conclusion}
In this paper, we studied a distributionally robust chance constrained problem (DRCCP) with Wasserstein ambiguity set. We showed that a DRCCP could be formulated as a conditional value-at-risk constrained optimization, thus admits tight inner and outer approximations. Once the feasible region is bounded, we showed that a DRCCP could be mixed integer representable with big-M coefficients and additional binary variables, i.e., a DRCCP can be formulated as a mixed integer conic program. We also compared various inner and outer approximations and proved their corresponding inclusive relations. We further proposed a big-M free formulation for a binary DRCCP and a branch and cut solution algorithm. The numerical studies demonstrated that the proposed formulations are quite promising.

\section*{Acknowledgments}  The author would like to thank Professor Shabbir Ahmed (Georgia Tech) for his helpful comments
on an earlier version of the paper. {Valuable comments from the editors and three anonymous reviewers are gratefully acknowledged.}

\bibliography{Reference}

\begin{thebibliography}{10}

\bibitem{ahmed2014nonanticipative}
S.~Ahmed, J.~Luedtke, Y.~Song, and W.~Xie.
\newblock Nonanticipative duality, relaxations, and formulations for
  chance-constrained stochastic programs.
\newblock {\em Mathematical Programming}, 162(1-2):51--81, 2017.

\bibitem{atamturk2008polymatroids}
A.~Atamt{\"u}rk and V.~Narayanan.
\newblock Polymatroids and mean-risk minimization in discrete optimization.
\newblock {\em Operations Research Letters}, 36(5):618--622, 2008.

\bibitem{bertsimas2018data}
D.~Bertsimas, S.~Shtern, and B.~Sturt.
\newblock A data-driven approach for multi-stage linear optimization.
\newblock Available at Optimization Online, 2018.

\bibitem{bertsimas2019twostage}
D.~Bertsimas, S.~Shtern, and B.~Sturt.
\newblock Two-stage sample robust optimization.
\newblock {\em arXiv preprint arXiv:1907.07142}, 2019.

\bibitem{blanchet2018distributionally}
J.~Blanchet, L.~Chen, and X.~Y. Zhou.
\newblock Distributionally robust mean-variance portfolio selection with
  {W}asserstein distances.
\newblock {\em arXiv preprint arXiv:1802.04885}, 2018.

\bibitem{blanchet2016robust}
J.~Blanchet, Y.~Kang, and K.~Murthy.
\newblock Robust {W}asserstein profile inference and applications to machine
  learning.
\newblock {\em arXiv preprint arXiv:1610.05627}, 2016.

\bibitem{blanchet2016quantifying}
J.~Blanchet and K.~R. Murthy.
\newblock Quantifying distributional model risk via optimal transport.
\newblock {\em arXiv preprint arXiv:1604.01446}, 2016.

\bibitem{calafiore2006scenario}
G.~C. Calafiore and M.~C. Campi.
\newblock The scenario approach to robust control design.
\newblock {\em IEEE Transactions on Automatic Control}, 51(5):742--753, 2006.

\bibitem{calafiore2006distributionally}
G.~C. Calafiore and L.~El~Ghaoui.
\newblock On distributionally robust chance-constrained linear programs.
\newblock {\em Journal of Optimization Theory and Applications}, 130(1):1--22,
  2006.

\bibitem{campi2009scenario}
M.~C. Campi, S.~Garatti, and M.~Prandini.
\newblock The scenario approach for systems and control design.
\newblock {\em Annual Reviews in Control}, 33(2):149--157, 2009.

\bibitem{chen2010cvar}
W.~Chen, M.~Sim, J.~Sun, and C.-P. Teo.
\newblock From {CVaR} to uncertainty set: Implications in joint
  chance-constrained optimization.
\newblock {\em Operations research}, 58(2):470--485, 2010.

\bibitem{chen2018data}
Z.~Chen, D.~Kuhn, and W.~Wiesemann.
\newblock Data-driven chance constrained programs over {W}asserstein balls.
\newblock {\em arXiv preprint arXiv:1809.00210}, 2018.

\bibitem{cheng2014distributionally}
J.~Cheng, E.~Delage, and A.~Lisser.
\newblock Distributionally robust stochastic knapsack problem.
\newblock {\em SIAM Journal on Optimization}, 24(3):1485--1506, 2014.

\bibitem{duan2018distributionally}
C.~Duan, W.~Fang, L.~Jiang, L.~Yao, and J.~Liu.
\newblock Distributionally robust chance-constrained approximate {AC-OPF} with
  {W}asserstein metric.
\newblock {\em IEEE Transactions on Power Systems}, 33(5):4924--4936, 2018.

\bibitem{edmonds1970submodular}
J.~Edmonds.
\newblock Submodular functions, matroids, and certain polyhedra.
\newblock In {\em Combinatorial Optimization-Eureka, You Shrink!}, pages
  11--26. Springer, 2003.

\bibitem{el2003worst}
L.~El~Ghaoui, M.~Oks, and F.~Oustry.
\newblock Worst-case value-at-risk and robust portfolio optimization: {A} conic
  programming approach.
\newblock {\em Operations Research}, 51(4):543--556, 2003.

\bibitem{esfahani2015data}
P.~M. Esfahani and D.~Kuhn.
\newblock Data-driven distributionally robust optimization using the
  {W}asserstein metric: Performance guarantees and tractable reformulations.
\newblock {\em Mathematical Programming}, 171(1-2):115--166, 2018.

\bibitem{fournier2015rate}
N.~Fournier and A.~Guillin.
\newblock On the rate of convergence in {W}asserstein distance of the empirical
  measure.
\newblock {\em Probability Theory and Related Fields}, 162(3-4):707--738, 2015.

\bibitem{gao2017wasserstein}
R.~Gao, X.~Chen, and A.~J. Kleywegt.
\newblock {W}asserstein distributional robustness and regularization in
  statistical learning.
\newblock {\em arXiv preprint arXiv:1712.06050}, 2017.

\bibitem{gao2016distributionally}
R.~Gao and A.~J. Kleywegt.
\newblock Distributionally robust stochastic optimization with {W}asserstein
  distance.
\newblock {\em arXiv preprint arXiv:1604.02199}, 2016.

\bibitem{gao2017distributionally}
R.~Gao and A.~J. Kleywegt.
\newblock Distributionally robust stochastic optimization with dependence
  structure.
\newblock {\em arXiv preprint arXiv:1701.04200}, 2017.

\bibitem{hanasusanto2016conic}
G.~A. Hanasusanto and D.~Kuhn.
\newblock Conic programming reformulations of two-stage distributionally robust
  linear programs over {W}asserstein balls.
\newblock {\em Operations Research}, 66(3):849--869, 2018.

\bibitem{hanasusanto2015distributionally}
G.~A. Hanasusanto, V.~Roitch, D.~Kuhn, and W.~Wiesemann.
\newblock A distributionally robust perspective on uncertainty quantification
  and chance constrained programming.
\newblock {\em Mathematical Programming}, 151:35--62, 2015.

\bibitem{hanasusanto2015Ambiguous}
G.~A. Hanasusanto, V.~Roitch, D.~Kuhn, and W.~Wiesemann.
\newblock Ambiguous joint chance constraints under mean and dispersion
  information.
\newblock {\em Operations Research}, 65(3):751--767, 2017.

\bibitem{hota2018data}
A.~R. Hota, A.~Cherukuri, and J.~Lygeros.
\newblock Data-driven chance constrained optimization under {W}asserstein
  ambiguity sets.
\newblock {\em arXiv preprint arXiv:1805.06729}, 2018.

\bibitem{ji2018data}
R.~Ji and M.~Lejeune.
\newblock Data-driven distributionally robust chanceconstrained optimization
  with {W}asserstein metric.
\newblock Avaiable at Optimization Online, 2018.

\bibitem{jiang2013data}
R.~Jiang and Y.~Guan.
\newblock Data-driven chance constrained stochastic program.
\newblock {\em Mathematical Programming}, 158:291--327, 2016.

\bibitem{kiesel2016wasserstein}
R.~Kiesel, R.~R{\"u}hlicke, G.~Stahl, and J.~Zheng.
\newblock The wasserstein metric and robustness in risk management.
\newblock {\em Risks}, 4(3):32, 2016.

\bibitem{lee2017minimax}
J.~Lee and M.~Raginsky.
\newblock Minimax statistical learning and domain adaptation with {W}asserstein
  distances.
\newblock {\em arXiv preprint arXiv:1705.07815}, 2017.

\bibitem{li2016ambiguous}
B.~Li, R.~Jiang, and J.~L. Mathieu.
\newblock Ambiguous risk constraints with moment and unimodality information.
\newblock {\em Mathematical Programming}, 173(1):151--192, Jan 2019.

\bibitem{lovasz1983submodular}
L.~Lov{\'a}sz.
\newblock Submodular functions and convexity.
\newblock In {\em Mathematical Programming The State of the Art}, pages
  235--257. Springer, 1983.

\bibitem{luedtke2008sample}
J.~Luedtke and S.~Ahmed.
\newblock A sample approximation approach for optimization with probabilistic
  constraints.
\newblock {\em SIAM Journal on Optimization}, 19(2):674--699, 2008.

\bibitem{luo2017decomposition}
F.~Luo and S.~Mehrotra.
\newblock Decomposition algorithm for distributionally robust optimization
  using {W}asserstein metric.
\newblock {\em arXiv preprint arXiv:1704.03920}, 2017.

\bibitem{nemirovski2006convex}
A.~Nemirovski and A.~Shapiro.
\newblock Convex approximations of chance constrained programs.
\newblock {\em SIAM Journal on Optimization}, 17(4):969--996, 2006.

\bibitem{nemirovski2006scenario}
A.~Nemirovski and A.~Shapiro.
\newblock Scenario approximations of chance constraints.
\newblock In {\em Probabilistic and randomized methods for design under
  uncertainty}, pages 3--47. Springer, 2006.

\bibitem{qiu2014covering}
F.~Qiu, S.~Ahmed, S.~S. Dey, and L.~A. Wolsey.
\newblock Covering linear programming with violations.
\newblock {\em INFORMS Journal on Computing}, 26(3):531--546, 2014.

\bibitem{rockafellar2000optimization}
R.~T. Rockafellar and S.~Uryasev.
\newblock Optimization of conditional value-at-risk.
\newblock {\em Journal of risk}, 2:21--42, 2000.

\bibitem{shafieezadeh2015distributionally}
S.~Shafieezadeh-Abadeh, P.~M. Esfahani, and D.~Kuhn.
\newblock Distributionally robust logistic regression.
\newblock In {\em Advances in Neural Information Processing Systems}, pages
  1576--1584, 2015.

\bibitem{song14chance}
Y.~Song, J.~R. Luedtke, and S.~K{\"u}{\c{c}}{\"u}kyavuz.
\newblock Chance-constrained binary packing problems.
\newblock {\em INFORMS Journal on Computing}, 26(4):735--747, 2014.

\bibitem{topkis1978minimizing}
D.~M. Topkis.
\newblock Minimizing a submodular function on a lattice.
\newblock {\em Operations research}, 26(2):305--321, 1978.

\bibitem{xie2019tractable}
W.~Xie.
\newblock Tractable reformulations of distributionally robust two-stage
  stochastic programs with $\infty-${W}asserstein distance.
\newblock Available at Optimization Online, 2018.

\bibitem{xie2016opf}
W.~Xie and S.~Ahmed.
\newblock Distributionally robust chance constrained optimal power flow with
  renewables: A conic reformulation.
\newblock {\em IEEE Transactions on Power Systems}, 33(2):1860--1867, 2018.

\bibitem{Xie2016drccp}
W.~Xie and S.~Ahmed.
\newblock On deterministic reformulations of distributionally robust joint
  chance constrained optimization problems.
\newblock {\em SIAM Journal on Optimization}, 28(2):1151--1182, 2018.

\bibitem{Xie2018approx}
W.~Xie and S.~Ahmed.
\newblock Bicriteria approximation of chance constrained covering problems.
\newblock Operations Research, 2019.

\bibitem{xie2017optimized}
W.~Xie, S.~Ahmed, and R.~Jiang.
\newblock Optimized bonferroni approximations of distributionally robust joint
  chance constraints.
\newblock {\em Available at Optimization Online}, 2017.

\bibitem{yang2014distributionally}
W.~Yang and H.~Xu.
\newblock Distributionally robust chance constraints for non-linear
  uncertainties.
\newblock {\em Mathematical Programming}, 155:231--265, 2016.

\bibitem{yu2017polyhedral}
J.~Yu and S.~Ahmed.
\newblock Polyhedral results for a class of cardinality constrained submodular
  minimization problems.
\newblock {\em Discrete Optimization}, 24:87--102, 2017.

\bibitem{zhang2016drccbp}
Y.~Zhang, R.~Jiang, and S.~Shen.
\newblock Ambiguous chance-constrained binary programs under mean-covariance
  information.
\newblock {\em SIAM Journal on Optimization}, 28(4):2922--2944, 2018.

\bibitem{zhao2015data_b}
C.~Zhao and Y.~Guan.
\newblock Data-driven risk-averse two-stage stochastic program with
  $\zeta$-structure probability metrics.
\newblock Available at
  \url{http://www.optimization-online.org/DB_FILE/2015/07/5014.pdf}, 2015.

\bibitem{zou2017stochastic}
J.~Zou, S.~Ahmed, and X.~A. Sun.
\newblock Stochastic dual dynamic integer programming.
\newblock {\em Mathematical Programming}, 175(1):461--502, 2019.

\bibitem{zymler2013distributionally}
S.~Zymler, D.~Kuhn, and B.~Rustem.
\newblock Distributionally robust joint chance constraints with second-order
  moment information.
\newblock {\em Mathematical Programming}, 137:167--198, 2013.

\end{thebibliography}

\end{document}